\documentclass[11pt]{article}
\usepackage[dvips]{graphicx}
\usepackage[dvips,letterpaper,margin=1in]{geometry}
\usepackage{color}
\usepackage{amsmath}
\usepackage{amsfonts}
\usepackage{amssymb}
\usepackage{amsthm}
\usepackage{newlfont}
\usepackage{epsfig, tikz}
\usetikzlibrary{arrows,snakes,backgrounds, matrix}
\usepackage{multirow}
\usepackage{setspace}
\usepackage{hyperref}
\usepackage{cite}
\usepackage{wrapfig}

\begin{document}

\newtheorem{assumption}{Assumption}[section]
\newtheorem{definition}{Definition}[section]
\newtheorem{lemma}{Lemma}[section]
\newtheorem{proposition}{Proposition}[section]
\newtheorem{theorem}{Theorem}[section]
\newtheorem{corollary}{Corollary}[section]
\newtheorem{remark}{Remark}[section]
\newtheorem{conjecture}{Conjecture}[section]
\newtheorem{example}{Example}[section]

\small

\title{A computational approach to persistence, permanence, and endotacticity of biochemical reaction systems}
\author{Matthew D. Johnston$^a$, Casian Pantea$^b$, and Pete Donnell$^c$ \bigskip \\
${}^a$ Department of Mathematics,\\
University of Wisconsin-Madison, Madison, WI 53706, USA\\
${}^b$ Department of Mathematics,\\
West Virginia University, Morgantown, WV 26506, USA\\
${}^c$ Department of Mathematics,\\
University of Portsmouth, Portsmouth, PO1 3HF, UK}
\date{}
\maketitle



\tableofcontents

\begin{abstract}
\small
We introduce a mixed-integer linear programming (MILP) framework capable of determining whether a chemical reaction network possesses the property of being endotactic or strongly endotactic. The network property of being strongly endotactic is known to lead to persistence and permanence of chemical species under genetic kinetic assumptions, while the same result is conjectured but as yet unproved for general endotactic networks. The algorithms we present are the first capable of verifying endotacticity of chemical reaction networks for systems with greater than two constituent species. We implement the algorithms in the open-source online package {\sf CoNtRol} and apply them to several well-studied biochemical examples, including the general $n$-site phosphorylation / dephosphorylation networks and a circadian clock mechanism.
\end{abstract}

\noindent \textbf{Keywords:} chemical reaction network, chemical kinetics, persistence, permanence, endotactic \newline \textbf{AMS Subject Classifications:} 92C42,  90C35 

\bigskip

\section{Introduction}
\label{introduction}

Significant attention has been paid recently to the question of how to extract qualitative dynamical information from topological properties of a reaction network's underlying graph of interacting species. One specific dynamical property of interest is {\it persistence}, whereby no initially present species may tend toward zero. This corresponds to non-extinction of a system's constituent species and is mathematically understood as trajectories avoiding $\partial \mathbb{R}_{> 0}^m$. Determining the persistence of chemical reaction systems is complicated by a number of factors: (a) the proposed mathematical models may be studied with a variety of kinetic assumptions, (b) parameter values are often unknown or known to only limited precision, and (c) the systems may have many interacting species.

It is surprising, therefore, that there are many classical results capable of characterizing the persistence, steady state, and stability properties of chemical reaction systems regardless of parameter values or even the kinetic form. An especially successful framework for such results has been that of {\it chemical reaction network theory} (CRNT). In the seminal paper \cite{H-J1}, Fritz Horn and Roy Jackson showed that every complex-balanced mass action system has a unique positive steady state within each minimally invariant affine subspace of the state space (called a {\em stoichiometric compatibility class}), and that this state is locally asymptotically stable. In \cite{F1} and \cite{H}, Horn and Martin Feinberg further relate the property of complex-balancing to a network parameter known as the ``deficiency''. The connection between the deficiency and steady state properties of mass action systems has been further developed in the papers \cite{F2,Fe2,Fe3,Fe4}.


In \cite{H-J1}, however, Horn and Jackson were unable to prove that initially positive trajectories far away from the complex-balanced equilibria could not tend toward $\partial \mathbb{R}_{>0}^m$ rather than the stoichiometrically-compatible positive steady state. This persistence property has been of interest more generally within chemical kinetics and also biology as it corresponds to the property that no initially present species (chemical or organism) tends toward extinction. Persistence of complex-balanced systems is known to be sufficient to prove the following \cite{S-M,C-D-S-S}:
\begin{conjecture}[Global Attractor Conjecture]
\label{gac}
The unique complex-balanced steady state in each positive stoichiometric compatibility class of a complex-balanced mass action system is a global attractor of its respective compatibility class.
\end{conjecture}
\noindent This conjecture, which is considered one of the most important open questions in CRNT, is known to hold in several special cases, including when the stoichiometric compatibility class is three-dimensional or less (Casian Pantea, \cite{Pa}) and when the network consists of only a single linkage class (David F. Anderson, \cite{A4}). The full conjecture remains open.


It was noted by Feinberg in \cite{F2} (Remark 6.1.E) that, although weakly reversible mass action systems which are not complex-balanced may permit varied behavior in the strictly positive region, including multistationarity and steady state instability, trajectories still do not approach $\partial \mathbb{R}_{> 0}^m$. He conjectured the following, which was adapted to its current form by Gheorghe Craciun {\it et al.} in \cite{C-N-P}:
\begin{conjecture}[Persistence Conjecture]
\label{wrc}
Every weakly reversible mass action system is persistent.
\end{conjecture}
\noindent What is striking about Conjecture \ref{wrc} is that, unlike the assumption of complex-balancing in Conjecture \ref{gac}, the condition of weak reversibility does not depend on the rate parameters. Attempts to affirm Conjecture \ref{wrc}, which would affirm Conjecture \ref{gac} as a corollary, have subsequently been {\it geometric} in nature. Specifically, weakly reversible mass action systems have the property that, whenever a trajectory approaches a face of $\mathbb{R}_{>0}^m$, the reaction vector corresponding to the largest reaction rate must necessarily point ``away'' from the face. This is commonly thought, on geometric grounds, to be sufficient to divert trajectories away from the boundary for all time.

It was further noted by Craciun {\it et al.} in \cite{C-N-P} that this geometric approach was not restricted to weakly reversible mass action systems; in fact, it does not depend on the reaction graph having direct connections at all. The authors made the following extensions: (1) weak reversibility was generalized to {\it endotacticity} of reaction networks, where the reaction vectors are required to point ``inward'' but are not required to connect in cycles; (2) mass action kinetics was generalized to {\it $\kappa$-variable} kinetics, where the rate parameters are allowed to vary with time within a compact region away from zero; and (3) persistence was strengthened to {\it permanence}, whereby there is a compact set bounded away from $\partial \mathbb{R}_{> 0}^m$ such that all initially positive trajectories are eventually contained with the set. They conjectured the following:
\begin{conjecture}[Permanence Conjecture]
\label{conj}
Every endotactic $\kappa$-variable mass action system is permanent.
\end{conjecture}
\noindent Proving Conjecture \ref{conj}, which is sufficient to prove Conjectures \ref{wrc} and \ref{gac} as corollaries, has subsequently become a focus of research. Conjecture \ref{conj} was shown by Craciun {\it et al.} in \cite{C-N-P} to hold for two-dimensional systems, a result which was extended by Pantea in \cite{Pa} to networks with a two-dimensional stoichiometric subspaces. The stronger notion of a network being {\it strongly endotactic} was introduced by Manoj Gopalkrishnan {\it et al.} in \cite{Gopal2014}. The authors there were able to affirm Conjecture \ref{conj} for all strongly endotactic networks.

Despite the recent work on endotactic networks and its known connections with several widely-studied open problems in CRNT, to date there does not exist a comprehensive method by which to determine whether a given reaction network is or is not endotactic. The strongest result along these lines is contained in the original paper \cite{C-N-P}, where the authors give necessary and sufficient conditions for the endotacticity of {\it two species networks only}. The method requires enumerating the faces of a particular convex polytope and then testing a condition on each such face. While a generalization of this result to higher dimensions appears to be a natural extension, the approach would be difficult to implement in practice and computationally-intensive for many species problems.

In this paper, we take an alternative approach to algorithmizing the process of determining whether a network is or is not endotactic. We formulate the problem within the framework of mixed-integer linear programming (MILP), which has already been used extensively in CRNT for determining dynamically equivalent and linearly conjugate realizations of a given kinetics satisfying a variety of structural constraints, including weak reversibility (G\'{a}bor Szederk\'{e}nyi {\it et al.}, \cite{Sz2,Sz-H-T,J-S4}). The algorithm we present here is able to definitively determine whether or not a network is endotactic or strongly endotactic.

We furthermore implement the algorithm in the open source software program {\sf CoNtRol} introduced by Pete Donnell {\it et al.} \cite{D-B-M-P} and use it to establish permanence of several many species networks which have been studied in the biochemical literature, including the $n$-site processive and distributive phosphorylation / dephosphorylation networks modeled with the quasi-steady state assumption, and a well-studied circadian clock mechanism (see Leloup and Goldbeter \cite{L-G} and Pokhilko {\it et al.} \cite{Pokhilko2010,Pokhilko2013}). This latter network was determined to be a network of interest through wide-scale application of the algorithm on the European Bioinformatics Institute's BioModels Database \cite{BioModels}.

\section{Background}


\subsection{Chemical reaction networks}
\label{crnsection}

The basic object of study in CRNT is a triple $(\mathcal{S},\mathcal{C},\mathcal{R})$ known as a {\it chemical reaction network}, where the sets are as follows:
\begin{enumerate}
\item[(1)]
The {\it species set} $\mathcal{S} = \{ X_1, \ldots, X_m \}$ consists of the basic species/reactants capable of undergoing chemical change.
\item[(2)]
The {\it complex set} $\mathcal{C} = \{ y_1, \ldots, y_n \}$ consists of formal linear combinations of species of the form
\[y_i = \sum_{j=1}^m y_{ij} X_j, \; \; \; \; \; i=1,\ldots,n\]
where the constants $y_{ij} \in \mathbb{R}$ are called {\it stoichiometric coefficients}. Allowing for a slight abuse of notation, we will let $y_i$ denote both the complex itself and the corresponding support vector $y_i = (y_{i1},y_{i2},\ldots,y_{im})$ which determines the multiplicity of each species within the corresponding complex. The complex with zero support is called the {\em zero complex} and denoted by $y_i = 0$. 
\item[(3)]
The {\it reaction set} $\mathcal{R} = \{ R_1, \ldots, R_r \}$ consists of elementary reactions of the form
\[R_k: \; \; \; y_{\rho(k)} \longrightarrow y_{\rho'(k)}, \; \; \; k=1, \ldots, r\]
where $\rho : \mathcal{R} \mapsto \mathcal{C}$ and $\rho' : \mathcal{R} \mapsto \mathcal{C}$. We say $y_{i}$ is the {\it source complex} of the $k^{th}$ reaction if $\rho(k)=i$. Similarly, we say $y_{j}$ is the {\it target complex} of the $k^{th}$ reaction if $\rho'(k) = j$. Reactions may alternatively be represented as ordered pairs of complexes, e.g. $R_k = (y_i,y_j)$ if $y_i \to y_j \in \mathcal{R}$.
\end{enumerate}

It is common within the CRNT literature to assume that (i) every species is in the support of at least one complex, (ii) every complex is involved in at least one reaction, and (iii) there are no self-loop reactions (i.e. reactions of the form $y_i \to y_i$). It is also common to assume the stoichiometric coefficients only take integer values. Our consideration of projected and single species networks in Section \ref{endotacticnetworkssection}, however, will force us to take some liberties with these assumptions. In particular, we will allow self-loops and non-integer stoichiometric coefficients.




Chemical reaction networks naturally give rise to a directed graph $G(V,E)$ where the vertices are given by the complexes (i.e. $V = \mathcal{C}$) and the edges are given by the reactions (i.e. $E = \mathcal{R}$). The connected components of $G$ are called {\em linkage classes}. We say that there is a {\it path} from complex $C_i$ to complex $y_j$ (denoted $y_i \leadsto y_j$) if there exists a sequence of complexes such that $y_i = y_{\mu(1)} \rightarrow y_{\mu(2)} \rightarrow \cdots \rightarrow y_{\mu(l)} = y_j$. A network is said to be {\it reversible} if $y_i \to y_j$ implies $y_j \to y_i$, and {\it weakly reversible} if $y_i \leadsto y_j$ implies $y_j \leadsto y_i$. To each reaction $R_k$, $k=1,\ldots, r$, we can associate the {\it reaction vector} $y_{\rho'(k)} - y_{\rho(k)} \in \mathbb{R}^m$. We define the {\it stoichiometric subspace} to be $S = \mbox{span} \{ y_{\rho'(k)} - y_{\rho(k)} \; | \; k=1,\ldots, r \}$ and assign $s = $ dim$(S)$.

For example, the network
\begin{equation}
\label{toyexample}
\begin{array}{c}
\displaystyle{X_1 \; \mathop{\stackrel{1}{\rightleftarrows}}_{2} \; 2X_2} \vspace{-0.05in} \\
\displaystyle{{}_{4} \nwarrow \; \; \; \swarrow_{3}} \; \; \\
X_2 +X_3 \; \;
\end{array}
\end{equation}
is weakly reversible but not reversible, since $2X_2 \to X_2 + X_3$ and $X_2 + X_3 \not\to 2X_2$, but $X_2 + X_3 \to X_1 \to 2X_2$ so that $X_2 + X_3 \leadsto 2X_2$. We have the reaction vectors $y_{\rho'(1)}-y_{\rho(1)} = (-1,2,0)$, $y_{\rho'(2)}-y_{\rho(2)} = (1,-2,0)$, $y_{\rho'(3)}-y_{\rho(3)} = (0,-1,1)$, and $y_{\rho'(4)}-y_{\rho(4)} = (1,-1,-1)$. We can easily determine that $s= \mbox{dim}(S) = 2$.

\subsection{$\kappa$-variable mass action systems}
\label{desection}

In order to determine how the concentrations $x_i = [X_i]$, $i=1,\ldots,m$, evolve over time, a {\it rate function} must be defined for each reaction. It is common to assume that the network obeys mass action kinetics whereby the rate of each reaction is proportional to the product of the concentrations of the reactant species. For instance, if the reaction is of the form $X_1 + X_2 \to \cdots$ then we would have rate $=\kappa [X_1][X_2] = \kappa x_1x_2$.

An extension of mass action kinetics is {\it $\kappa$-variable mass action kinetics}, whereby the rate parameters are allowed to vary in time within a compact region bounded away from zero \cite{C-N-P}. In this framework, we instead have the rate $=\kappa(t)x_1x_2$ where $\kappa(t) \in [\eta, 1/\eta]$, $\eta > 0$. This generalization allows for changes in the reaction rate due to external factors like changes in heat and pressure and is also useful for model and dimension reduction techniques \cite{A4,Gopal2014}. Other commonly used kinetics forms include Michaelis-Menten kinetics and Hill kinetics \cite{M-M,Hi}, which we will consider further in Section \ref{applicationsection}.

Throughout this paper, we will assume that the evolution of the concentration vector $\mathbf{x} = (x_1, \ldots, x_m) \in \mathbb{R}_{\geq 0}^m$ is given by the {\it $\kappa$-variable mass action system}
\begin{equation}
\label{de}
\frac{d\mathbf{x}(t)}{dt} = \sum_{k=1}^r \kappa_k(t) (y_{\rho'(k)}-y_{\rho(k)}) \mathbf{x}(t)^{y_{\rho(k)}}
\end{equation}
where $\mathbf{x}(t)^{y_{\rho(k)}} = \prod_{i=1}^m x_i(t)^{y_{\rho(k)i}}$.
It is clear that mass action kinetics is a special case of $\kappa$-variable mass action kinetics, taking $\kappa_k(t) = \kappa_k$ for all $k=1, \ldots, r$. Also note that $d\mathbf{x}/dt \in S$ for all $t \geq 0$; since, moreover, all initially positive solutions of (\ref{de}) stay positive, solutions are restricted to {\it stoichiometric compatibility classes}, $(\mathbf{a} + S) \cap \mathbb{R}_{> 0}^m$ where $\mathbf{a} \in \mathbb{R}^m$. That is, if $\mathbf{x}_0 \in (\mathbf{a} + S) \cap \mathbb{R}_{> 0}^m$ then $\mathbf{x}(t) \in (\mathbf{a} + S) \cap \mathbb{R}_{> 0}^m$ for $t \geq 0$.

\subsection{Persistence and permanence}
\label{persistencesection}

It is often desirable to guarantee that no coordinate of a solution $\mathbf{x}(t)$ of (\ref{de}) approaches zero, either asymptotically or along a subsequence, as this corresponds to a species approaching extinction. This behavior is captured by the notion of {\em persistence} \cite{persistWaltman, hopfbauerPerm, LV}:
\begin{definition}
\label{persistence}
We say that a trajectory $\mathbf{x}(t)$ of (\ref{de}) with $\mathbf{x}(0) \in \mathbb{R}_{>0}^m$ is \textbf{persistent} if
\[\liminf_{t \to \infty} x_i(t) > 0 \; \; \; \mbox{ for all } i=1,\ldots,m.\]
The system itself is said to be \textbf{persistent} if every trajectory $\mathbf{x}(t)$ with $\mathbf{x}(0) \in \mathbb{R}_{>0}^m$ is persistent.
\end{definition}
\noindent Persistence is closely tied to the following classical notation from dynamical systems theory.
\begin{definition}
\label{omegalimitset}
We say that $\mathbf{x}^* \in \mathbb{R}_{\geq 0}^m$ is an \textbf{$\omega$-limit point} of a trajectory $\mathbf{x}(t)$ of (\ref{de}) is there exists a subsequence of times $t_k$, $k=1,\ldots,m$, such that
\[\lim_{k \to \infty} t_k = \infty \; \; \; \mbox{ and } \; \; \; \lim_{k \to \infty} \mathbf{x}(t_k) = \mathbf{x}^*.\]
The set of all \textbf{$\omega$-limit points} of a given trajectory $\mathbf{x}(t)$ with initial condition $\mathbf{x}(0)=\mathbf{x}_0$ is denoted $\omega(\mathbf{x}_0)$.
\end{definition}
\noindent If the trajectories $\mathbf{x}(t)$ of (\ref{de}) are known to be bounded, then the system is persistent if and only if $\omega(\mathbf{x}(0)) \cap \partial\mathbb{R}_{>0}^m = \O$ for all $\mathbf{x}(0) \in \mathbb{R}_{> 0}^m$.

While persistence does not allow trajectories to approach $\partial \mathbb{R}_{>0}^m$, it does not preclude the possibility of trajectories staying close to $\partial \mathbb{R}_{>0}^m$ if they start close. It was noted by Craciun {\it et al.} in \cite{C-N-P} that many networks of interest exhibit {\em permanent} behavior \cite{hopfbauerPerm, LV} which limits how trajectories starting close to $\partial \mathbb{R}_{> 0}^m$ behave.

\begin{definition}
\label{permanent} We say that the system (\ref{de}) is \textbf{permanent} if, for every positive stoichiometric compatibility class $(\mathbf{a} + S) \cap \mathbb{R}_{> 0}^m$, there exists an $\epsilon > 0$ such that, for every $\mathbf{x}_0 \in (\mathbf{a} + S) \cap \mathbb{R}_{> 0}^m$, the solution ${\mathbf x}(t)$ of (\ref{de}) with 
${\mathbf x}(0)={\mathbf x}_0$ satisfies
\[\liminf_{t \to \infty} x_i(t) > \epsilon \; \; \; \mbox{ and } \; \; \; \limsup_{t \to \infty} x_i(t) < 1/\epsilon \; \; \; \mbox{ for all } i=1,\ldots,m.\]
\end{definition}
\noindent That is, trajectories must eventually converge to a stoichiometrically compatible compact region which is strictly buffered from $\partial \mathbb{R}_{>0}^m$. To highlight the distinction between persistence and permanence, consider the following example.

\begin{example}
\label{lotkavolterra}
Consider the following chemical reaction network:
\begin{equation}
\label{lv}
\left\{ \begin{array}{l}
X \; \stackrel{k_1}{\longrightarrow} \; 2X \\
X + Y \; \stackrel{k_2}{\longrightarrow} \; 2Y \\ 
Y \; \stackrel{k_3}{\longrightarrow} \; 0
\end{array} \right.
\end{equation}
This network (\ref{lv}) corresponds to the well-studied Lotka-Volterra predator-prey model with $X$ as the prey and $Y$ as the predator. Assuming mass action kinetics with fixed rate constants, the trajectories of the kinetic system (\ref{de}) are cycles which follow
\begin{equation}
\label{levelcurve}
L(x(t),y(t)) = k_3 \ln(x(t)) + k_1 \ln(y(t)) - k_2(x(t)+y(t)) = \mbox{constant}.
\end{equation}
Since the level sets of this function never intersect $\partial\mathbb{R}_{>0}^m$, it follows that the network is persistent; however, the system is not permanent since, for any $\epsilon > 0$, we may pick an initial condition so that $x_i(0) < \epsilon$ and are guaranteed that the trajectory will return to this state infinitely often. We may interpret this as saying that, although the trajectory does not tend toward extinction, it is not able to buffer itself from  extinction.


\end{example}

\section{Endotactic networks}
\label{endotacticnetworkssection}

It has become common in recent years to approach the question of persistence and permanence of chemical kinetic systems (\ref{de}) as one of geometry rather than one of algebra. Note that each reaction in (\ref{de}) gives a term of the form
\begin{equation}
\label{term}
\left[ \kappa_k(t) \mathbf{x}(t)^{y_{\rho(k)}} \right] \cdot (y_{\rho'(k)} - y_{\rho(k)})
\end{equation}
where $\kappa_k(t) \mathbf{x}(t)^{y_{\rho(k)}}$ is a scalar for all $t \geq 0$ and $\mathbf{x} \in \mathbb{R}_{> 0}^m$. Near $\partial \mathbb{R}_{> 0}^m$, this coefficient will be small for any reaction which depends on a species which is close to extinction. The corresponding reaction vector $(y_{\rho'(k)} - y_{\rho(k)})$ will not factor significantly in (\ref{de}). It is intuitive to suppose, therefore, that if the reaction with the dominant rate near a particular boundary region corresponds to a reaction vector which directs solutions away from boundary, then the solutions may not approach that boundary region over time.


One classification of systems where this desired property holds trivially is for weakly reversible mass action systems. (See Conjecture \ref{wrc} in Section \ref{introduction}.) In \cite{C-N-P}, Craciun {\it et al.} identified a broader class of reaction networks for which the dominating rates were guarantee to correspond to reactions which were pointing away from the boundary. They called such networks {\it endotactic networks}, a notion which later inspired the class of {\it strongly endotactic networks} introduced by Gopalkrishnan {\it et al.} in \cite{Gopal2014}.


\begin{definition}
\label{endotactic}
A chemical reaction network is said to be:
\begin{enumerate}
\item
\textbf{endotactic} if, for every $\mathbf{w} \in \mathbb{R}^m$ and every $R_i \in \mathcal{R}$, $\mathbf{w} \cdot (y_{\rho'(i)} - y_{\rho(i)}) < 0$ implies that there is a $R_j \in \mathcal{R}$ such that $\mathbf{w} \cdot (y_{\rho(j)} - y_{\rho(i)}) < 0$ and $\mathbf{w} \cdot (y_{\rho'(j)} - y_{\rho(j)}) > 0$.
\item
\textbf{lower endotactic} if condition $1.$ holds for every $\mathbf{w} \in \mathbb{R}^m_{\geq 0}$ instead of $\mathbf{w} \in \mathbb{R}^m$.
\item
\textbf{strongly endotactic} if condition $1.$ holds and $R_j \in \mathcal{R}$ can be chosen so that $\mathbf{w} \cdot (y_{\rho(k)} - y_{\rho(j)}) \geq 0$ for all $R_k \in \mathcal{R}$.
\end{enumerate}
\end{definition}

\noindent Definition \ref{endotactic} differs from that given in the papers \cite{C-N-P,Pa,Gopal2014}, but it can easily be checked that our formulation is equivalent. It is also clear that if a network is not lower endotactic, it is not endotactic, and if it is not endotactic, then it is not strongly endotactic, i.e. the following inclusions hold:
\[\mbox{strongly endotactic} \; \subseteq \; \mbox{endotactic} \; \subseteq \; \mbox{lower endotactic}.\]


\subsection{Single-species and projected networks}
\label{projectedsection}

To verify whether a network is endotactic, it is useful to represent complexes as points in $\mathbb{R}^m$ and reactions as arrows between these points (Figure \ref{fig1}). We may then determine whether a network is endotactic or not by conducting a ``parallel sweep test'' \cite{C-N-P}. To conduct this test, we sweep through the space of complexes $\mathbb{R}^m$ with hyperplanes orthogonal to a representative sample of vectors $\mathbf{w} \in \mathbb{R}^m$. In order to be endotactic, it must be the case that for every sweeping direction, whenever a reaction vectors points against this direction, we have already ``swept through'' a reaction from a source complex which points strictly in the same direction. In order to be strongly endotactic, it must furthermore be the case that there is a reaction which points strictly in the same direction which comes from the first set of source complexes which was swept through.
\begin{center}
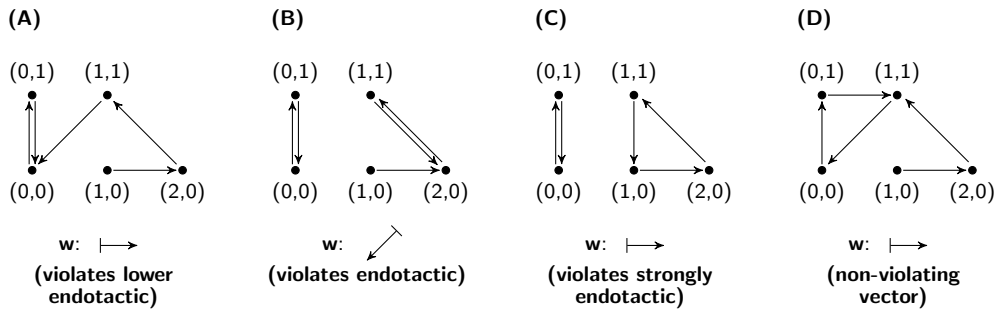
\begin{figure}[t]
\begin{tikzpicture}
\tikzstyle{stable}=[circle,draw=black,fill=black, minimum size=1mm, inner sep=0pt]
\tikzstyle{unstable}=[circle,draw=black,fill=white, minimum size=1mm, inner sep=0pt]
\node {};
\begin{scope}[xshift=2cm]
\begin{scope}[xshift=0cm]
\node [stable] at (0,0) {};
\node [stable] at (1,0) {};
\node [stable] at (1,-1) {};
\node [stable] at (0,-1) {};
\node [stable] at (2,-1) {};
\draw[->, >=stealth'] (0.04,-0.08) -- (0.04,-0.92);
\draw[->, >=stealth'] (-0.04,-0.92) -- (-0.04,-0.08);
\draw[->, >=stealth'] (0.92,-0.08) -- (0.08,-0.92);
\draw[->, >=stealth'] (1.08,-1) -- (1.92,-1);
\draw[->, >=stealth'] (1.92,-0.92) -- (1.08,-0.08);
\draw[|->, >=stealth'] (0.9,-2) -- (1.4,-2);
\node at (0.5,-2) {\scriptsize\sf{\bf w}:};
\node at (0.95,-2.4) {\scriptsize\sf{\bf (violates lower  }};
\node at (0.95,-2.7) {\scriptsize\sf{\bf   endotactic)}};
\end{scope}
\begin{scope}[xshift=3.5cm]
\node [stable] at (0,0) {};
\node [stable] at (1,0) {};
\node [stable] at (1,-1) {};
\node [stable] at (0,-1) {};
\node [stable] at (2,-1) {};
\draw[->, >=stealth'] (0.04,-0.08) -- (0.04,-0.92);
\draw[->, >=stealth'] (-0.04,-0.92) -- (-0.04,-0.08);
\draw[->, >=stealth'] (1.05,-0.09) -- (1.89,-0.93);
\draw[->, >=stealth'] (1.08,-1) -- (1.92,-1);
\draw[->, >=stealth'] (1.95,-0.89) -- (1.11,-0.05);
\draw[|->, >=stealth'] (1.35,-1.8) -- (0.95,-2.2);
\node at (0.5,-2.0) {\scriptsize\sf{\bf w}:};
\node at (0.95,-2.4) {\scriptsize\sf{\bf (violates endotactic)}};
\node at (0.95,-2.7) {\scriptsize\sf{\bf }};
\end{scope}
\begin{scope}[xshift=7cm]
\node [stable] at (0,0) {};
\node [stable] at (1,0) {};
\node [stable] at (1,-1) {};
\node [stable] at (0,-1) {};
\node [stable] at (2,-1) {};
\draw[->, >=stealth'] (0.04,-0.08) -- (0.04,-0.92);
\draw[->, >=stealth'] (-0.04,-0.92) -- (-0.04,-0.08);
\draw[->, >=stealth'] (1,-0.08) -- (1,-0.92);
\draw[->, >=stealth'] (1.08,-1) -- (1.92,-1);
\draw[->, >=stealth'] (1.95,-0.89) -- (1.11,-0.05);
\draw[|->, >=stealth'] (0.9,-2) -- (1.4,-2);
\node at (0.5,-2) {\scriptsize\sf{\bf w}:};
\node at (0.95,-2.4) {\scriptsize\sf{\bf (violates strongly}};
\node at (0.95,-2.7) {\scriptsize\sf{\bf endotactic)}};
\end{scope}
\begin{scope}[xshift=10.5cm]
\node [stable] at (0,0) {};
\node [stable] at (1,0) {};
\node [stable] at (1,-1) {};
\node [stable] at (0,-1) {};
\node [stable] at (2,-1) {};
\draw[->, >=stealth'] (0.08,0) -- (0.92,0);
\draw[->, >=stealth'] (0,-0.92) -- (0,-0.08);
\draw[->, >=stealth'] (0.92,-0.08) -- (0.08,-0.92);
\draw[->, >=stealth'] (1.08,-1) -- (1.92,-1);
\draw[->, >=stealth'] (1.95,-0.89) -- (1.11,-0.05);
\draw[|->, >=stealth'] (0.9,-2) -- (1.4,-2);
\node at (0.5,-2) {\scriptsize\sf{\bf w}:};
\node at (0.95,-2.4) {\scriptsize\sf{\bf (non-violating}};
\node at (0.95,-2.7) {\scriptsize\sf{\bf vector)}};
\end{scope}
\begin{scope}
\scriptsize\sf
\node at (0,0.3) {(0,1)};
\node at (1,0.3) {(1,1)};
\node at (0,-1.3) {(0,0)};
\node at (1,-1.3) {(1,0)};
\node at (2,-1.3) {(2,0)};
\node at (3.5,0.3) {(0,1)};
\node at (4.5,0.3) {(1,1)};
\node at (3.5,-1.3) {(0,0)};
\node at (4.5,-1.3) {(1,0)};
\node at (5.5,-1.3) {(2,0)};
\node at (7,0.3) {(0,1)};
\node at (8,0.3) {(1,1)};
\node at (7,-1.3) {(0,0)};
\node at (8,-1.3) {(1,0)};
\node at (9,-1.3) {(2,0)};
\node at (10.5,0.3) {(0,1)};
\node at (11.5,0.3) {(1,1)};
\node at (10.5,-1.3) {(0,0)};
\node at (11.5,-1.3) {(1,0)};
\node at (12.5,-1.3) {(2,0)};
\node at (-.1,1) {\bf(A)};
\node at (3.4,1) {\bf(B)};
\node at (6.9,1) {\bf(C)};
\node at (10.4,1) {\bf(D)};
\end{scope}
\end{scope}
\end{tikzpicture}
\caption{\small Four example networks with complexes and reactions represented by their support vectors in $\mathbb{R}^2$. For example, the edge $(2,0) \to (1,1)$ present in all four networks corresponds to the reaction $2X_1 \to X_1 + X_2$. By the parallel sweep test, we have that (A) is not lower endotactic, (B) is lower endotactic but not endotactic, (C) is endotactic but not strongly endotactic, and (D) is strongly endotactic. Where relevant, the violating vectors $\mathbf{w}$ are indicated below the network.}\label{fig1}
\end{figure}
\end{center}

\vspace{-1.02cm}
It is worth noting that, for a given $\mathbf{w} \in \mathbb{R}^m$, the action of sweeping through the source complexes in $\mathbb{R}^m$ is reduced to a one-dimensional problem. In fact, we may always correspond such a problem to an analogous network problem involving only a single species. We will call such networks {\it single species networks}. We also introduce the following.

\begin{definition}\label{def:proj}
Consider a chemical reaction network and a unit vector $\mathbf w \in \mathbb{R}^m$. Let $\pi_{\mathbf w}: \mathbb{R}^{m} \mapsto \mathbb{R}$ denote the orthogonal projection onto $\mathbf w$, i.e. $\pi_{\mathbf{w}}(\mathbf{y}) = \mathbf{w} \cdot \mathbf{y}$ for $\mathbf{y} \in \mathbb{R}^m$. We define the \textbf{$\mathbf w$-projected network} $\pi_{\mathbf w}(\cal R)$
to be the single species network with complexes ${\cal C}_{\mathbf w}=\{\pi_{\mathbf w}(y) \; | \; y\in {\cal C}\}$ and reactions $\mathcal{R}_{\mathbf{w}}=\{\pi_{\mathbf w}(y)\to \pi_{\mathbf w}(y') \; | \; y\to y'\in {\cal R}\}.$ 
\end{definition}

It is worth noting that $\mathbf w$-projected networks may contain self-loop reactions and have non-integer stoichiometric coefficients. This is typically disallowed in CRNT but will be permitted in our present context. We define the {\em proper subnetwork} of a single species network to be the network where all self-loop reactions and source complexes involved in only self-loop reactions have been removed. Given this allowance, the proper subnetwork may be an empty network; however, if the proper subnetwork is not empty, then it contains at least two complexes since a network may have a single complex if and only if it consists of one self-loop reaction.


Establishing whether a single species network is endotactic or strongly endotactic can be easily done geometrically. In this case the complexes are identified with real numbers, and in Definition \ref{endotactic} there are only two directions for the vector ${\mathbf w}\in \mathbb R$, namely ${\mathbf w}=1$ and  ${\mathbf w}=-1$. It follows that a single species network is endotactic if and only if 

\begin{enumerate}
\item[($+$)]  whenever $y\to y'$ is a reaction and $y<y'$ there exists a reaction ${\bar y}\to {\bar y'}$ with ${\bar y}>{\bar y}'$ and ${\bar y}>y$; and
\item[($-$)]  whenever $y\to y'$ is a reaction and $y>y'$ there exists a reaction ${\bar y}\to {\bar y'}$ with ${\bar y}<{\bar y}'$ and ${\bar y}<y.$
\end{enumerate} 
\noindent A network is furthermore strongly endotactic if the two $\bar{y}$ above may be chosen to satisfy: $(+)$ ${\bar y} \geq \tilde{y}$ for source complexes $\tilde{y}$, and $(-)$ ${\bar y} \leq \tilde{y}$ for all source complexes $\tilde{y}$.

It is vacuously true that any empty network is endotactic and strongly endotactic, and that a single species / single complex network is endotactic and strongly endotactic if and only if it consists only of self-loops. We may characterize single species networks with two or more source complexes by establishing two {\it extreme source complexes} which yield all other source complexes as proper convex combinations. We have:


\begin{itemize}
\item[-]
\noindent {\em Strongly endotactic}: The network is strongly endotactic if and only if there is a reaction from each extreme source complex which points in the direction of the other source complexes. 
\item[-]
\noindent {\em Endotactic:} The network is endotactic if and only if its proper subnetwork is strongly endotactic.
\end{itemize}

In view of Definition \ref{def:proj} and this discussion, we may also characterize endotacticity and strong endotacticity for networks with arbitrary number of species in the following way. 

\begin{proposition}\label{prop:lineNetw} A reaction network is endotactic (strongly endotactic) if and only if, for every $\mathbf{w} \in \mathbb{R}^m$, the corresponding $\mathbf{w}$-projected network is endotactic (strongly endotactic).
\end{proposition}
\begin{proof} From Definition \ref{endotactic}, a network is not endotactic if and only if there is some ${\mathbf w}\in{\mathbb R}^m$ and some reaction 
$R_i \in \mathcal{R}$ such that  $\mathbf{w} \cdot (y_{\rho'(i)} - y_{\rho(i)}) < 0$ and for any  $R_j \in \mathcal{R}$ with $\mathbf{w} \cdot (y_{\rho(j)} - y_{\rho(i)}) < 0$ we have $\mathbf{w} \cdot (y_{\rho'(j)} - y_{\rho(j)}) \le 0$. Letting ${y} =\pi_{\mathbf w}(y_{\rho(i)})$, ${y'} =\pi_{\mathbf w}(y_{\rho'(i)})$, 
${\bar y} =\pi_{\mathbf w}(y_{\rho(j)})$, ${\bar y'} =\pi_{\mathbf w}(y_{\rho'(j)}),$ this is equivalent to saying that
there is ${\mathbf w}\in{\mathbb R}^m$ such that in the $\mathbf{w}$-projected network $\pi_{\mathbf w}({\cal R})$ there is a reaction $y\to y'$ with $y'<y$ such that for any reaction ${\bar y}\to {\bar y'}$ with ${\bar y}<y$ we have ${\bar y}\ge {\bar y'},$ i.e. 
\begin{itemize}
\item[(*)] there is ${\mathbf w}\in{\mathbb R}^m$ such that in the $\mathbf{w}$-projected network $\pi_{\mathbf w}({\cal R})$ condition ($+$) is violated.  
\end{itemize}

Condition (*) implies that $\pi_{\mathbf w}({\cal R})$ is not endotactic. Conversely, if $\pi_{\mathbf w}({\cal R})$ is not endotactic for some ${\mathbf w}\in{\mathbb R}^m$ then either ($+$) or ($-$) above are violated in $\pi_{\mathbf w}({\cal R})$. As seen above, the first case amounts to (*) above and is equivalent to the original network not being endotactic. The second case is equivalent to ($+$) being false in $\pi_{-\mathbf w}({\cal R})$, which implies (*) and the fact that the original network is not endotactic.

The proof of the strongly endotactic case is similar.
\end{proof}

In Figure \ref{fig:1dEndo}, we present the $\mathbf{w}$-projected networks (I) and corresponding proper subnetworks (II) of the four networks given in Figure \ref{fig1}. We can see that the single species networks (A) and (B) in (I) are not endotactic since the left extremal complex does not have a reaction which points to the right. We can also see that the networks (C) and (D) are endotactic because their proper subnetworks in (II) are strongly endotactic. The single-species network (C), however, is not strongly endotactic because there is no reaction from the extremal left source complex which points to the right. It follows from Proposition \ref{prop:lineNetw} that (A) is not lower endotactic, (B) is not endotactic, and (C) is not strongly endotactic. It is worth noting that affirming endotacticity or strong endotacticity by Proposition \ref{prop:lineNetw} is more challenging, as this requires constructing $\mathbf{w}$-projected networks like those in Figure \ref{fig:1dEndo} for a representative sample of vectors $\mathbf{w} \in \mathbb{R}^m$.

\begin{center}
\begin{figure}[t]
\begin{tikzpicture}
\tikzstyle{stable}=[circle,draw=black,fill=black, minimum size=1mm, inner sep=0pt]
\tikzstyle{unstable}=[circle,draw=black,fill=white, minimum size=1mm, inner sep=0pt]
\node {};
\begin{scope}[xshift=2cm]
\begin{scope}[yshift=-0.5cm]
\node [stable] at (0,0) {};
\node [stable] at (1,0) {};
\node [stable] at (2,0) {};
\draw [->,>=stealth'] (0.05,0.05) .. controls (0.7,.5) and (-0.7,.5) .. (-0.05,0.05);
\draw [->,>=stealth'] (0.92,0) -- (0.08,0);
\draw [->,>=stealth'] (1.08,0.04) -- (1.92,0.04);
\draw [->,>=stealth'] (1.92,-0.04) -- (1.08,-0.04);
\end{scope}
\begin{scope}[yshift=-0.5cm, xshift=3.5cm]
\node [stable] at (0,0) {};
\node [stable] at (1,0) {};
\node [stable] at (2,0) {};
\draw [->,>=stealth'] (0.05,0.05) .. controls (0.7,.5) and (-0.7,.5) .. (-0.05,0.05);
\draw [->,>=stealth'] (0.92,0) -- (0.08,0);
\draw [->,>=stealth'] (1.08,0.04) -- (1.92,0.04);
\draw [->,>=stealth'] (1.92,-0.04) -- (1.08,-0.04);
\end{scope}
\begin{scope}[yshift=-0.5cm, xshift=7cm]
\node [stable] at (0,0) {};
\node [stable] at (1,0) {};
\node [stable] at (2,0) {};
\draw [->,>=stealth'] (0.05,0.05) .. controls (0.7,.5) and (-0.7,.5) .. (-0.05,0.05);
\draw [->,>=stealth'] (1.05,0.05) .. controls (1.7,.5) and (0.3,.5) .. (0.95,0.05);
\draw [->,>=stealth'] (1.08,0.04) -- (1.92,0.04);
\draw [->,>=stealth'] (1.92,-0.04) -- (1.08,-0.04);
\end{scope}
\begin{scope}[yshift=-0.5cm, xshift=10.5cm]
\node [stable] at (0,0) {};
\node [stable] at (1,0) {};
\node [stable] at (2,0) {};
\draw [->,>=stealth'] (0.05,0.05) .. controls (0.7,.5) and (-0.7,.5) .. (-0.05,0.05);
\draw [->,>=stealth'] (0.08,0.04) -- (0.92,0.04);
\draw [->,>=stealth'] (0.92,-0.04) -- (0.08,-0.04);
\draw [->,>=stealth'] (1.08,0.04) -- (1.92,0.04);
\draw [->,>=stealth'] (1.92,-0.04) -- (1.08,-0.04);
\end{scope}
\begin{scope}[yshift=-2.5cm]
\node [unstable] at (0,0) {};
\node [stable] at (1,0) {};
\node [stable] at (2,0) {};
\draw [->,>=stealth'] (0.92,0) -- (0.08,0);
\draw [->,>=stealth'] (1.08,0.04) -- (1.92,0.04);
\draw [->,>=stealth'] (1.92,-0.04) -- (1.08,-0.04);
\end{scope}
\begin{scope}[yshift=-2.5cm, xshift=3.5cm]
\node [unstable] at (0,0) {};
\node [stable] at (1,0) {};
\node [stable] at (2,0) {};
\draw [->,>=stealth'] (0.92,0) -- (0.08,0);
\draw [->,>=stealth'] (1.08,0.04) -- (1.92,0.04);
\draw [->,>=stealth'] (1.92,-0.04) -- (1.08,-0.04);
\end{scope}
\begin{scope}[yshift=-2.5cm, xshift=7cm]
\node [stable] at (1,0) {};
\node [stable] at (2,0) {};
\draw [->,>=stealth'] (1.08,0.04) -- (1.92,0.04);
\draw [->,>=stealth'] (1.92,-0.04) -- (1.08,-0.04);
\end{scope}
\begin{scope}[yshift=-2.5cm, xshift=10.5cm]
\node [stable] at (0,0) {};
\node [stable] at (1,0) {};
\node [stable] at (2,0) {};
\draw [->,>=stealth'] (0.08,0.04) -- (0.92,0.04);
\draw [->,>=stealth'] (0.92,-0.04) -- (0.08,-0.04);
\draw [->,>=stealth'] (1.08,0.04) -- (1.92,0.04);
\draw [->,>=stealth'] (1.92,-0.04) -- (1.08,-0.04);
\end{scope}
\begin{scope}
\footnotesize\sf
\node at (1.2,0.4) {\bf (I) w-Projected Networks:};
\node at (1.1,-2) {\bf (II) Proper Subnetworks:};
\scriptsize\sf
\node at (0,-0.9) {(0,1)};
\node at (0,-1.3) {(0,0)};
\node at (1,-0.9) {(1,1)};
\node at (1,-1.3) {(1,0)};
\node at (2,-0.9) {(2,0)};
\node at (3.5,-0.9) {(1,1)};
\node at (3.5,-1.3) {(2,0)};
\node at (4.5,-0.9) {(0,1)};
\node at (4.5,-1.3) {(1,0)};
\node at (5.5,-0.9) {(0,0)};
\node at (7,-0.9) {(0,1)};
\node at (7,-1.3) {(0,0)};
\node at (8,-0.9) {(1,1)};
\node at (8,-1.3) {(1,0)};
\node at (9,-0.9) {(2,0)};
\node at (10.5,-0.9) {(0,1)};
\node at (10.5,-1.3) {(0,0)};
\node at (11.5,-0.9) {(1,1)};
\node at (11.5,-1.3) {(1,0)};
\node at (12.5,-0.9) {(2,0)};
\end{scope}
\end{scope}
\end{tikzpicture}
\vspace{0.1cm}
\caption{\small In (I) we show the $\mathbf{w}$-projected networks for the four networks (A), (B), (C) and (D) and respective vectors $\mathbf{w} \in \mathbb{R}^2$ given in Figure \ref{fig1}. From left to right, the orders of the source complexes in the $\mathbf{w}$-projected networks correspond to the orders in which $\mathbf{w}$ sweeps through the these complexes in the original networks. In (II) we have removed the self-loops and complexes only associated with self-loops from (I) to get the corresponding proper subnetwork.}\label{fig:1dEndo}
\end{figure}
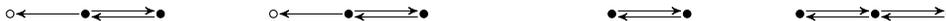
\end{center}

\subsection{Known results}
\label{knownresultssection}

While the conjectures given in Section \ref{introduction} (Conjecture \ref{gac}, \ref{wrc}, and \ref{conj}) remain open, several special cases are known where they hold. 

\begin{theorem}[Theorem 5.1, \cite{Pa}]
\label{theorem1}
Every $\kappa$-variable mass action system with bounded trajectories, a two-dimensional stoichiometric subspace, and lower-endotactic stoichiometric subnetworks is persistent.
\end{theorem}



\begin{theorem}[Theorem 1.1, \cite{Gopal2014}]
\label{theorem2}
Every strongly endotactic mass action system is permanent.
\end{theorem}

\noindent Notably, the first result is limited by dimension and requires that the system have bounded trajectories, which may not be immediately known. The second result is not limited by dimension or the boundedness of trajectories. It is worth noting, however, that many endotactic networks fail to be strongly endotactic (for example, the weakly reversible network (C) in Figure \ref{fig1}(I)).

\section{Main results}
\label{mainsection}

We now consider the question of determining whether a network is lower endotactic, endotactic, strongly endotactic, or none of the above. For the algorithms we will introduce in the remainder of this section, we use the following equivalent characterization of Definition \ref{endotactic}. This result is proved in the Supplemental Material.

\begin{lemma}
\label{endolemma}
A chemical reaction network is not endotactic if and only if there is a non-zero vector $\mathbf{w}= (w_1,w_2,\ldots, w_m) \in \mathbb{R}^m$ and a partition of the reaction set $\mathcal{R}$ into three disjoint sets $\mathcal{R}_0$, $\mathcal{R}_-$, and $\mathcal{R}_+$ such that:
\begin{enumerate}
\item[E1.]
$\mathbf{w} \cdot (y_{\rho'(i)} - y_{\rho(i)}) = 0$ for all $R_i \in \mathcal{R}_0$;
\item[E2.]
$\mathbf{w} \cdot (y_{\rho'(i)} - y_{\rho(i)}) < 0$ for all $R_i \in \mathcal{R}_-$;
\item[E3.]
$\mathbf{w} \cdot (y_{\rho(i)} - y_{\rho(j)}) \leq 0$ for all $R_i \in \mathcal{R}_-$ and $R_j \in \mathcal{R}_+$; and
\item[E4.]
$\mathcal{R}_- \not= \O$.
\end{enumerate}
A network is not lower endotactic if and only if there is a non-zero $\mathbf{w} \in \mathbb{R}^m$ and a partition of $\mathcal{R}$ such that E1, E2, E3, and E4 are satisfied as well as:
\begin{enumerate}
\item[LE.]
$\mbox{w}_k \geq 0$ for all $k = 1, \ldots, m$.
\end{enumerate}
A network is not strongly endotactic if and only if there is a non-zero $\mathbf{w} \in \mathbb{R}^m$ and a partition of $\mathcal{R}$ such that E1, E2, and E4 are satisfied as well as:
\begin{enumerate}
\item[SE3(a).]
$\mathcal{R}_0 \not= \O$ and $\mathbf{w} \cdot (y_{\rho(i)} - y_{\rho(j)}) < 0$ for all $R_i \in \mathcal{R}_0$ and $R_j \in \mathcal{R}_- \cup \mathcal{R}_+$; or
\item[SE3(b).]
$\mathcal{R}_0 = \O$ and $\mathbf{w} \cdot (y_{\rho(i)} - y_{\rho(j)}) \leq 0$ for all $R_i \in \mathcal{R}_-$ and $R_j \in \mathcal{R}_+$; and
\end{enumerate}
\end{lemma}

Notice that Lemma \ref{endolemma} gives necessary and sufficient conditions for a network to {\it not} be endotactic. This is advantageous for the computational procedures we will outline in this section since checking endotacticity directly requires checking the technical conditions of Definition \ref{endotactic} of Proposition \ref{prop:lineNetw} {\it for all} vectors $\mathbf{w} \in \mathbb{R}^m$. Lemma \ref{endolemma}, however, requires showing that {\it there exists} a vector $\mathbf{w} \in \mathbb{R}^m$ satisfying the given conditions. Linear programming problems are well-suited to answer questions of existence.

\subsection{Mixed-integer linear programming framework}
\label{milpsection}

A {\it mixed-integer linear programming} (MILP) problem with $k$ decision variables (which we will denote in vector form $\mathbf{y} \in \mathbb{R}^k$) and $p$ constraints may be written in the form
\begin{equation}
\label{milp}
\begin{split}
\mbox{minimize} \; \; &\mathbf{c} \cdot \mathbf{y} \\
\mbox{subject to} \; \; & A_1 \mathbf{y} = \mathbf{b}_1 \\ 
& A_2 \mathbf{y} \leq \mathbf{b}_2 \\
& y_j \mbox{ is an integer} \\
& \mbox{for } j \in I, I \subseteq \{ 1, \ldots, k\}
\end{split}
\end{equation}
where $\mathbf{c} \in \mathbb{R}^k$, $A_1 \in \mathbb{R}^{p_1 \times k}$, $A_2 \in \mathbb{R}^{p_2 \times k}$, and $p_1 + p_2 = p$ \cite{Sz2}. If all of the variables in the problem are real then (\ref{milp}) can be solved in polynomial time. If any of the variables are integer-valued, however, it becomes NP-hard. The development of algorithms for efficiently solving MILP problems is a major area of current work which we do not summarize here. We utilize the non-commercial software packages GNU Linear Program Kit (GLPK) \cite{Makhorin2010} and SCIP \cite{Achterberg2009}. Several problems within the scope of network identification have already been placed in the MILP framework (Szederk\'{e}nyi {\it et al.}, \cite{Sz2,Sz-H-T,J-S4}).

We now recast the conditions of Lemma \ref{endolemma} for determining whether or not a network is endotactic into a MILP framework. The following quantities must be initialized prior to the start of the procedure:
\begin{equation}
\tag{\textbf{Par}}
\label{parameters}
\left\{
\begin{split}
& \mbox{A {\it stoichiometric matrix} }\Gamma \in \mathbb{Z}^{m \times r}\mbox{ where }[\Gamma]_{\cdot,k} = y_{\rho'(k)} - y_{\rho(k)}.\hfill \\
& \mbox{A {\it source complex matrix} }Y \in \mathbb{Z}^{m \times r}\mbox{ where }[Y]_{\cdot,k} = y_{\rho(k)}.\\
& \mbox{A small parameter }\epsilon > 0.
\end{split} \right.
\end{equation}
If multiple reactions proceed from the same source complex, then that source complex will appear multiple times in $Y$ as distinct columns. This differs from the convention often adopted in CRNT for enumerating complexes in a {\it complex matrix} $Y$ where repeated complexes are removed. This difference will not have any consequence on the processes we will introduce.

\subsection{Endotactic and lower endotactic networks}
\label{milpendo}

We initialize $\Gamma \in \mathbb{Z}^{m \times r}$, $Y \in \mathbb{Z}_{\geq 0}^{m \times r}$, and $\epsilon > 0$ as in (\ref{parameters}). We introduce the following decision variables:
\begin{flalign}
\tag{\textbf{Var}}
\label{dec1}
&
\left\{ \begin{array}{l}
\displaystyle{\mbox{w}[i] \in \left[ -\frac{1}{\epsilon}, \frac{1}{\epsilon} \right], i=1, \ldots, m} \\
\displaystyle{\mbox{R}_0[j] \in \{ 0, 1\}, j=1,\ldots, r}\\
\displaystyle{\mbox{R}_-[j] \in \{ 0, 1\}, j=1,\ldots, r}
\end{array}
\right.
&
\end{flalign}
and wish to satisfy the following logical equivalences, where $\mathbf{w}$, $\mathcal{R}_0$, $\mathcal{R}_-$, and $\mathcal{R}_+$ are as in Lemma \ref{endolemma}:
\begin{flalign}
\label{requirements1}
&
\left\{
\begin{array}{l}
\mathbf{w}[i] = w_i \mbox{ for } \mathbf{w} = (w_1, \ldots, w_m)\\
\mbox{R}_0[j] = 1 \; \Longleftrightarrow \; R_j \in \mathcal{R}_0\\
\mbox{R}_-[j] = 1 \; \Longleftrightarrow \; R_j \in \mathcal{R}_-.
\end{array}
\right.
&
\end{flalign}
Notice that, since $\mathcal{R}_0$, $\mathcal{R}_-$, and $\mathcal{R}_+$ form a complete partition of $\mathcal{R}$, we may determine the set $\mathcal{R}_+$ by noting that $R_j \in \mathcal{R}_+$ if and only if $\mbox{R}_0[j] = 0$ and $\mbox{R}_-[j] =0$.

We can accommodate the conditions E1-E3 of Lemma \ref{endolemma} and the logical implications of (\ref{requirements1}) with the following constraint sets, where we take $i,j=1,\ldots,r$, $i \not= j$:
\begin{flalign}
\tag{\textbf{E1}}
\label{E1}
&
\left\{ \begin{array}{l} 
\displaystyle{\sum_{k=1}^m \mbox{w}[k] \cdot \Gamma_{k,i} \leq \frac{1}{\epsilon} \cdot (1 - \mbox{R}_0[i])}\\
\displaystyle{-\sum_{k=1}^m \mbox{w}[k] \cdot \Gamma_{k,i} \leq \frac{1}{\epsilon} \cdot (1 - \mbox{R}_0[i])}\\
 \end{array} \right.
&
\end{flalign}
\begin{flalign}
\tag{\textbf{E2}}
\label{E2}
&
\left\{ \begin{array}{l}
\displaystyle{\mbox{R}_-[i] \leq 1 - \mbox{R}_0[i]}\\
\displaystyle{\sum_{k=1}^m \mbox{w}[k] \cdot \Gamma_{k,i} \leq \frac{1}{\epsilon} \cdot(1-\mbox{R}_-[i]) - \epsilon}
\end{array} \right.
&
\end{flalign}
\begin{flalign}
\tag{\textbf{E3}}
\label{E3}
&
\left\{ \begin{array}{l}
\displaystyle{\sum_{k=1}^m \mbox{w}[k]\cdot(Y_{k,i}-Y_{k,j})\leq \frac{1}{\epsilon} \cdot(1-\mbox{R}_-[i] + \mbox{R}_-[j] +\mbox{R}_0[j])}\\
\end{array} \right.
&
\end{flalign}

We make a few notes about the constraint sets above. {\it (E1):} By the construction of $\Gamma$ and the binarity of $\mbox{R}_0$, this condition admits only two possibilities. If $\mbox{R}_0[i] = 1$ then $\mathbf{w} \cdot (y_{\rho'(i)} - y_{\rho(i)}) = 0$, and if $\mbox{R}_0[i] = 0$ then $-\frac{1}{\epsilon} \leq \mathbf{w} \cdot (y_{\rho'(i)} - y_{\rho(i)}) \leq \frac{1}{\epsilon}$, which is no constraint for small $\epsilon >0$. {\it (E2):} The first constraint guarantees that $\mbox{R}_0[i]$ and $\mbox{R}_-[i]$ may not simultaneously be one, so that $R_i \in \mathcal{R}_0$ and $R_i \in \mathcal{R}_-$ is not possible. The second contraint permits two possibilities. If $\mbox{R}_-[i] = 1$ then $\mathbf{w} \cdot (y_{\rho'(i)}-y_{\rho(i)}) \leq -\epsilon < 0$, and if $\mbox{R}_0[i] = 0$ then $\mathbf{w} \cdot (y_{\rho'(i)}-y_{\rho(i)}) \leq \frac{1}{\epsilon} - \epsilon$, which provides no restriction for small $\epsilon > 0$. {\it (E3):} Notice that $\mathbf{w} \cdot (y_i - y_j) \leq 0$ is guaranteed only if $\mbox{R}_-[i] = 1$, $\mbox{R}_-[j] = 0$, and $\mbox{R}_0[j] = 0$, which is equivalent to the case $R_i \in \mathcal{R}_-$ and $R_j \in \mathcal{R}_+$. Otherwise, the constraint becomes $\mathbf{w} \cdot (y_i - y_j) \leq \frac{n}{\epsilon}$, $n=1,2$, which is no restriction for small $\epsilon > 0$.

Notice that (\ref{E1}), (\ref{E2}), and (\ref{E3}) may be trivially satisfied by taking $w[i] = 0$ for $i=1,\ldots,m,$ and $\mbox{R}_-[j] = 0$ for $j=1, \ldots, r$. We want to construct the objective function in (\ref{milp}) to ensure that, if possible, a non-trivial $\mathbf{w}$ and non-empty set $\mathcal{R}_-$ is selected. We therefore introduce the following objective function:
\begin{equation}
\tag{\textbf{Endo}}
\label{endo}
\mbox{minimize} \; \; \; - \sum_{i=1}^r \mbox{R}_-[i].
\end{equation}
If the network is not endotactic, there is a $\mathbf{w} \not= \mathbf{0}$ which satisfies conditions E1 - E3 of Lemma \ref{endolemma}, so that $\mbox{R}_-[i] = 1$ for at least one $i=1,\ldots,r$. If the network is endotactic, however, no $\mathbf{w} \not= \mathbf{0}$ satisfying the conditions may be found so that the procedures returns $\mbox{R}_-[i] = 0,$ $i=1,\ldots,r$, and the trivial vector $\mathbf{w} = \mathbf{0}$. It follows that we may determine whether a chemical reaction network is endotactic or not by optimizing (\ref{endo}) over the constraint sets (\ref{E1}), (\ref{E2}), and (\ref{E3}). If the program returns the optimal value zero, then the network is endotactic; otherwise, the program returns a negative value, and the network is not endotactic.

If we wish to determine whether of not a network is simply lower endotactic, we can incorporate condition $LE$ from Lemma \ref{endolemma} with the following constraint set, where we take $k=1,\ldots, m$:
\begin{flalign}
\tag{\textbf{LE}}
\label{LE}
&
\Big\{ \begin{array}{l}
\displaystyle{-\mbox{w}[k] \leq 0.}\\
\end{array} \Bigg.
&
\end{flalign}
We may then determine whether the network is lower endotactic or not by optimizing (\ref{endo}) over the constraint sets (\ref{E1}), (\ref{E2}), (\ref{E3}), and (\ref{LE}). The interpretation remains the same as above.

\subsection{Strongly endotactic networks}
\label{milpsendo}

We now turn our attention to the conditions of Lemma \ref{endolemma} for determining whether or not a network is strongly endotactic. We initialize $\Gamma \in \mathbb{Z}^{m \times r}$, $Y \in \mathbb{Z}_{\geq 0}^{m \times r}$, and $\epsilon > 0$ as in (\ref{parameters}) and introduce the following decision variables:
\begin{flalign}
\tag{\textbf{SEVar}}
\label{dec2}
&
\left\{ \begin{array}{l}
\displaystyle{\mbox{w}[i] \in \left[ -\frac{1}{\epsilon}, \frac{1}{\epsilon} \right], i=1, \ldots, m} \\
\displaystyle{\mbox{R}_0[j] \in \{ 0, 1\}, j=1,\ldots, r} \\
\displaystyle{\mbox{R}_-[j] \in \{ 0, 1\}, j =1, \ldots, r} \\
\displaystyle{\Theta \in \{ 0, 1 \}}
\end{array}
\right.
&
\end{flalign}
We now wish to satisfy the following logical equivalences, where $\mathbf{w}$, $\mathcal{R}_-$, and $\mathcal{R}_+$ are as in Lemma \ref{endolemma}:
\begin{flalign}
\label{requirements2}
&
\left\{
\begin{array}{l}
\mathbf{w}[i] = w_i \mbox{ for } \mathbf{w} = (w_1, \ldots, w_m)\\
\mbox{R}_0[j] = 1 \; \Longleftrightarrow \; R_j \in \mathcal{R}_0\\
\mbox{R}_-[j] = 1 \; \Longleftrightarrow \; R_j \in \mathcal{R}_-\\
\Theta = 1 \; \Longleftrightarrow \; \mbox{SE}3(b) \mbox{ is required}
\end{array}
\right.
&
\end{flalign}
The variable $\Theta\in \{ 0, 1 \}$ determines whether we are considering condition SE$3(a)$ ($\Theta=0$) or SE$3(b)$ ($\Theta=1$) of Lemma \ref{endolemma}. Otherwise, the interpretation is the same as in Section \ref{milpendo}.

We can accommodate conditions SE3(a) and SE3(b) of Lemma \ref{endolemma} and the logical implications of (\ref{requirements2}) with the following constraint sets, where we take $i,j=1,\ldots,r$, $i \not= j$:
\begin{flalign}
\tag{\textbf{SE3a}}
\label{SE3a}
&
\left\{ \begin{array}{l}
\displaystyle{\sum_{k=1}^m \mbox{w}[k] \cdot (Y_{k,i}-Y_{k,j}) \leq \frac{1}{\epsilon} \cdot (1-\mbox{R}_-[i] + \mbox{R}_-[j]) - \epsilon}
\end{array} \right.
&
\end{flalign}
\begin{flalign}
\tag{\textbf{SE3b}}
\label{SE3b}
&
\left\{ \begin{array}{l}
\displaystyle{1 - \Theta \leq \frac{1}{\epsilon} \sum_{i=1}^r \mbox{R}_0[i]}\\
\displaystyle{\Theta - 1 \leq \frac{1}{\epsilon} \sum_{i=1}^r \mbox{R}_0[i]}\\
\displaystyle{\sum_{k=1}^m \mbox{w}[k] \cdot (Y_{k,i}-Y_{k,j}) \leq \frac{1}{\epsilon} \cdot ( 2 - \mbox{R}_-[i]+\mbox{R}_-[j]-\Theta).}
\end{array} \right.
&
\end{flalign}

We make a few notes about the constraint sets above. {\it (SE3a):} If $R_i \in \mathcal{R}_-$ and $R_j \in \mathcal{R}_0 \cup \mathcal{R}_-$ then we have $\mathbf{w} \cdot (y_{\rho(i)}-y_{\rho(j)}) \leq -\epsilon < 0$; otherwise, we have $\mathbf{w} \cdot (y_{\rho(i)}- y_{\rho(j)}) \leq \frac{n}{\epsilon} - \epsilon$, which is no restriction for small $\epsilon > 0$. {\it (SE3b):} If $\mathcal{R}_0 = \O$ then condition SE$3(a)$ may not be satisfied so that we must check SE$3(b)$. The first two constraints guarantee that $\Theta=1$. If $R_i \in \mathcal{R}_-$, $R_j \in \mathcal{R}_0 \cup \mathcal{R}_+$, and SE$3(a)$ cannot be satisfied, we have $\mathbf{w} \cdot (y_{\rho(i)}-y_{\rho(j)} \leq 0$; otherwise, we have $\mathbf{w} \cdot (y_{\rho(i)}-y_{\rho(j)} \leq \frac{n}{\epsilon}$, $n=1,2$, which is no restriction for small $\epsilon > 0$.

As in Section \ref{milpendo}, (\ref{E1}), (\ref{E2}), (\ref{SE3a}), and (\ref{SE3b}) may be trivially satisfied by taking $w[i] = 0$ for $i=1,\ldots,m,$ and $\mbox{R}_-[j] = 0$ for $j=1, \ldots, r$. To avoid this situation and satisfy condition E$4$, we reintroduce the objective function (\ref{endo}). We may determine whether a chemical reaction network is strongly endotactic or not by optimizing (\ref{endo}) over the constraint sets (\ref{E1}), (\ref{E2}), (\ref{SE3a}), and (\ref{SE3b}). If the program returns the value zero, then the network is strongly endotactic; otherwise, it is not strongly endotactic.

\section{Implementation and application}
\label{applicationsection}

We have implemented the MILP procedure outlined in Section \ref{milpsection} in the freely available web based chemical reaction network analysis tool {\sf CoNtRol}, which is available at \url{http://reaction-networks.net/control/} \cite{D-B-M-P}. In addition to determining whether the chemical reaction network is endotactic or strongly endotactic, the program implements several classical and recent results to test for multiple equilibria, periodic orbits, monotonicity, and convergence to stable equilibria.

We have run the {\sf CoNtRol} implementation on a large sample of networks from the European Bioinformatics Institute's BioModels Database \cite{BioModels}. The search produced a number of examples of networks which were either endotactic or strongly endotactic but which were not weakly reversible, and therefore would not have been previously classified. We note, however, that the majority of networks which were endotactic were also weakly reversible, and the majority of networks which were strongly endotactic were weakly reversible and consisted of a single linkage class. We suspect, therefore, that in application networks which are endotactic but outside of the scope of the theory of weakly reversible networks are rare.

In the remainder of this section, we consider examples drawn from the biochemical literature, including the database search. Firstly, we consider the general distributed and processive phosphorylation / dephosphorylation networks which are widely used to model signal transduction cascades (e.g. MAPK cascades). We show that these networks are not endotactic in their standard form but that they may be reduced through the standard quasi-steady state approximation (QSSA) to models which are strongly endotactic and have $\kappa$-variable mass action form. We may therefore conclude permanence. Secondly, we consider a common circadian clock mechanism \cite{L-G} which was a network of interest in our BioModels search. We show that this network is strongly endotactic and has $\kappa$-variable mass action form and is therefore permanent. It is notable that this network is not weakly reversible and therefore falls outside of the scope of networks which can be considered by existing methods. 

\subsection{Distributive and processive phosphorylation networks}




Protein substrates are commonly modified posttranslationally by the process of enzyme-mediated phosphorylation and dephosphorylation. In these processes, phosphate groups are attached to or removed from the protein substrate in order to change its shape and functions. It is common for the phosphorylation and dephosphorylation events to be mediated by distinct enzymes, for instance, as is the case for MAPK/ERK systems, and also for these proteins to permit multiple phosphorylation sites on the same substrate. (See the survey of Patwardhan and Miller for details \cite{Patwardhan}.)

We start by considering the following basic motif:
\[
\mbox{\textbf{(futile cycle)}} \; \; \left\{ \; \;\begin{array}{l} \displaystyle{S_0 + E \mathop{\stackrel{k_1}{\begin{array}{c} \vspace{-0.25cm} \longleftarrow \vspace{-0.3cm} \\ \longrightarrow \end{array}}}_{k_2} ES_0 \; \stackrel{k_3}{\longrightarrow} \; S_1 + E} \\[0.2in]
\displaystyle{S_1 + F \mathop{\stackrel{\ell_1}{\begin{array}{c} \vspace{-0.25cm} \longleftarrow \vspace{-0.3cm} \\ \longrightarrow \end{array}}}_{\ell_2} FS_1 \; \stackrel{\ell_3}{\longrightarrow} \; S_0 + F} \end{array} \right.
\]
where $S_0$ is an unphosphorylated substrate, $S_1$ is a phorphorylated substrate, $E$ and $F$ are distinct enzymes, and $ES_0$ and $FS_1$ are substrate-enzyme compounds. This model was originally studied by Goldbeter and Koshland \cite{Goldbeter1981} and is often referred to as the {\it futile cycle} due to the apparent conflicting objectives of the enzymes: $E$ seeks to phosphoryate $S$, while $F$ seeks to dephosphorylate it. Persistence of this mechanism has been shown previously in Angeli and Sontag \cite{A-S2} and Sontag {\it et al.} \cite{A3}. The endotacticity of this network has not previously been considered in the literature. We present the following.

\begin{proposition}
\label{futilecyclelemma}
The futile cycle network is not endotactic.
\end{proposition}

\begin{proof}
The {\sf CoNtRol} implementation of the algorithm in Section \ref{milpendo} establishes this result. For completeness, we present the results. We index the species according to $X_1 = S_0$, $X_2 = E$, $X_3 = ES_0$, $X_4 = S_1$, $X_5 = F$, and $X_6 = FS_1$, and the reactions in the order of the rates: $R_1 \leftrightarrow k_1$, $R_2 \leftrightarrow k_2$, $R_3 \leftrightarrow k_3$, $R_4 \leftrightarrow \ell_1$, $R_5 \leftrightarrow \ell_2$, and $R_6 \leftrightarrow \ell_3$. We initialize the following matrices according to (\ref{parameters}):
\[\Gamma = \left[ \begin{array}{cccccc} -1 & 1 & 0 & 0 & 0 & 1 \\ -1 & 1 & 1 & 0 & 0 & 0 \\ 1 & -1 & -1 & 0 & 0 & 0 \\ 0 & 0 & 1 & -1 & 1 & 0 \\ 0 & 0 & 0 & -1 & 1 &1 \\ 0 & 0 & 0 & 1 & -1 & -1 \end{array} \right] \hspace{1cm} \mbox{{\it and}} \hspace{1cm} Y = \left[ \begin{array}{cccccc} 1 & 0 & 0 & 0 & 0 & 0 \\ 1 & 0 & 0 & 0 & 0 &0 \\ 0 & 1 & 1 & 0 & 0 & 0 \\ 0 & 0 & 0 & 1 & 0 & 0 \\ 0 & 0 & 0 & 1 & 0 & 0 \\ 0 & 0 & 0 & 0 & 1 & 1 \end{array} \right]\]
and select $\epsilon = 0.1$. The algorithm returns the vector $\mathbf{w} = (9.9,0.1,10,0,10,10)$ and identifies the reaction $R_3 = (ES_0 \longrightarrow S_1 + E) \in \mathcal{R}_-$. It can be quickly checked that $\mathbf{w} \cdot Y_{\cdot, 1} =\mathbf{w} \cdot Y_{\cdot,2}  = \mathbf{w} \cdot Y_{\cdot,3} = \mathbf{w} \cdot Y_{\cdot,4} =\mathbf{w} \cdot Y_{\cdot,5} = \mathbf{w} \cdot Y_{\cdot,6} = 10$ and that $\mathbf{w} \cdot \Gamma_{\cdot, 3} = -9.9 < 0$. It follows that the $\mathbf{w}$-projected network consists of a single source complex and that $\mathcal{R}_3$ projects to a non-self-loop reaction. It follows that the network is not endotactic by Definition \ref{endotactic} and Proposition \ref{prop:lineNetw}.
\end{proof}


In many biochemical processes, substrates are phosphorylated multiple times rather than the single time permitted by the futile cycle. For such proteins, the phosphorylation and dephosphorylation processes may furthermore occur {\it processively} or {\it distributively}. The general $n$-site processive phosphorylation network is given by
\begin{equation}
\label{processive}
\left\{
\begin{array}{c}
\displaystyle{S_0 + E \mathop{\stackrel{k_1}{\begin{array}{c} \vspace{-0.25cm} \longleftarrow \vspace{-0.3cm} \\ \longrightarrow \end{array}}}_{k_2} ES_0 \mathop{\stackrel{k_3}{\begin{array}{c} \vspace{-0.25cm} \longleftarrow \vspace{-0.3cm} \\ \longrightarrow \end{array}}}_{k_4} ES_1 \mathop{\stackrel{k_5}{\begin{array}{c} \vspace{-0.25cm} \longleftarrow \vspace{-0.3cm} \\ \longrightarrow \end{array}}}_{k_6} \cdots \mathop{\stackrel{k_{2n-1}}{\begin{array}{c} \vspace{-0.25cm} \longleftarrow \vspace{-0.3cm} \\ \longrightarrow \end{array}}}_{k_{2n}} ES_{n-1} \; \stackrel{k_{2n+1}}{\longrightarrow} \; S_n + E} \vspace{0.1in} \\
\displaystyle{S_n + F \mathop{\stackrel{\ell_{2n+1}}{\begin{array}{c} \vspace{-0.25cm} \longleftarrow \vspace{-0.3cm} \\ \longrightarrow \end{array}}}_{\ell_{2n}} FS_n \mathop{\stackrel{\ell_{2n-1}}{\begin{array}{c} \vspace{-0.25cm} \longleftarrow \vspace{-0.3cm} \\ \longrightarrow \end{array}}}_{\ell_{2n-2}} \cdots \mathop{\stackrel{\ell_5}{\begin{array}{c} \vspace{-0.25cm} \longleftarrow \vspace{-0.3cm} \\ \longrightarrow \end{array}}}_{\ell_4} FS_2 \mathop{\stackrel{\ell_3}{\begin{array}{c} \vspace{-0.25cm} \longleftarrow \vspace{-0.3cm} \\ \longrightarrow \end{array}}}_{\ell_2} FS_1 \; \stackrel{\ell_1}{\longrightarrow} \; S_0 + F}
\end{array}
\right.
\end{equation}
while the $n$-site distributive phosphorylation network is given by
\begin{equation}
\label{distributive}
\left\{
\begin{split}
& \displaystyle{S_0 + E \mathop{\stackrel{k_1}{\begin{array}{c} \vspace{-0.25cm} \longleftarrow \vspace{-0.3cm} \\ \longrightarrow \end{array}}}_{k_2} S_0E \; \stackrel{k_3}{\longrightarrow} S_1 + E \mathop{\stackrel{k_4}{\begin{array}{c} \vspace{-0.25cm} \longleftarrow \vspace{-0.3cm} \\ \longrightarrow \end{array}}}_{k_5} S_1E \; \stackrel{k_6}{\longrightarrow} \; \; \; \cdots \; \; \; \stackrel{k_{3n-3}}{\longrightarrow} S_{n-1} + E \mathop{\stackrel{k_{3n-2}}{\begin{array}{c} \vspace{-0.25cm} \longleftarrow \vspace{-0.3cm} \\ \longrightarrow \end{array}}}_{k_{3n-1}} S_{n-1}E\stackrel{k_{3n}}{\longrightarrow} S_n + E} \\
&\displaystyle{S_n + F \mathop{\stackrel{\ell_{3n-2}}{\begin{array}{c} \vspace{-0.25cm} \longleftarrow \vspace{-0.3cm} \\ \longrightarrow \end{array}}}_{\ell_{3n-1}} S_nF \; \stackrel{\ell_{3n}}{\longrightarrow} S_{n-1} + F \mathop{\stackrel{\ell_{3n-5}}{\begin{array}{c} \vspace{-0.25cm} \longleftarrow \vspace{-0.3cm} \\ \longrightarrow \end{array}}}_{\ell_{3n-4}} S_{n-1}F \; \stackrel{\ell_{3n-3}}{\longrightarrow}\; \; \; \cdots \; \; \; \stackrel{\ell_6}{\longrightarrow} S_{1} + F \mathop{\stackrel{\ell_1}{\begin{array}{c} \vspace{-0.25cm} \longleftarrow \vspace{-0.3cm} \\ \longrightarrow \end{array}}}_{\ell_2} S_1F \; \stackrel{\ell_3}{\longrightarrow} S_0 + F}
\end{split}
\right.
\end{equation}
\noindent In (\ref{processive}) and (\ref{distributive}), $S_n$ denotes substrate which has been phosphorylated $n$ times. The distinction between (\ref{processive}) and (\ref{distributive}) lies in whether the facilitating enzyme detaches after every step or remains attached.

Dynamical properties of the distributive mechanism has been studied by numerous authors, including Gunawardena \cite{Gunawardena}, Wang and Sontag \cite{W-S}, Perez-Millan {\it et al.} \cite{M-D-S-C}, and Johnston \cite{J1}. The processive network has been studied by Thomson and Gunawardena \cite{T-G} and Conradi and Shiu \cite{C-S}. While persistence for the distributive mechanism (\ref{distributive}) is known from the results of Angeli {\it et al.} \cite{A3}, endotacticity has not previously be studied. We complete this study now.

\begin{proposition}
The distributive and processive phosphorylation / dephosphorylation networks given by (\ref{processive}) and (\ref{distributive}) are not endotactic.
\end{proposition}

\begin{proof} We construct a vector $\mathbf{w}$ similar to the one found computationally in Proposition \ref{futilecyclelemma}.\\

\noindent {\it Processive:} We index the species and reactions of (\ref{processive}) according to: $X_i = ES_{i-1}$, $i=1, \ldots, n$; $X_{n+i} = FS_{i}$, $i=1,\ldots, n$; $X_{2n+1} = S_0$; $S_{2n+2} = S_n$; $S_{2n+3} = E$; $S_{2n+4} = F$; $R_i \leftrightarrow k_i$, $i=1, \ldots, 2n+1$; and $R_{2n+1+i} \leftrightarrow \ell_i$, $i=1,\ldots, 2n+1$. We set-up the matrices $\Gamma$ and $Y$ as in (\ref{parameters}). Consider the following vector:
\begin{equation}
\label{w2}
\begin{array}{crccccccccccl}
& & ES_0 &\cdots & ES_{n-1} & FS_1 & \cdots & FS_n & S_0 & S_n & E & F & \\
\mathbf{w}= & \Big( & 1, & \cdots & 1, & 1, & \cdots & 1, & 0, & 1, & 1, & 0 & \Big).
\end{array}
\end{equation}
We have that $\mathbf{w} \cdot Y_{\cdot,i} =  1$, for all $i=1,\ldots, 4n+2$, so that the $\mathbf{w}$-projected network consists of a single source complex. For the reaction $R_{2n+2} = (FS_1 \longrightarrow F + S_0)$, however, we have $\mathbf{w} \cdot \Gamma_{\cdot,2n+2} = -1 < 0$ so that $R_{2n+2}$ does not project to a self-loop reaction. It follows that (\ref{processive}) is not endotactic by Definition \ref{endotactic} and Proposition \ref{prop:lineNetw}.\\

\noindent {\it Distributive:} We index the species and reactions of (\ref{distributive}) according to: $X_i = ES_{i-1}$, $i=1, \ldots, n$; $X_{n + i} = FS_{i}$, $i=1,\ldots, n$; $X_{2n+i} = S_{i-1}$, $i=1, \ldots, n+1$; $X_{3n+1} = E$; $X_{3n+2} = F$; $R_i \leftrightarrow k_i$, $i=1,\ldots, 3n$; and $R_{3n + i} \leftrightarrow \ell_{i}$, $i=1, \ldots, 3n$. We set-up the matrices $\Gamma$ and $Y$ as in (\ref{parameters}). Consider the following vector:
\begin{equation}
\label{w1}
\begin{array}{crccccccccccccl}
& & ES_0 &\cdots & ES_{n-1} & FS_1 & \cdots & FS_n & S_0 & S_1 &\cdots & S_n & E & F & \\
\mathbf{w}= & \Big( & 1, & \cdots & n, & 1, & \cdots & n, & 0, & 1, & \cdots & n, & 1, & 0 & \Big).
\end{array}
\end{equation}
The first source complexes encountered when sweeping with the $\mathbf{w}$ given by (\ref{w1}) are $S_0 + E$, $S_1 + F$, $ES_0$, and $FS_1$, which return a minimal value of $\mathbf{w} \cdot Y_{\cdot,1} = \mathbf{w} \cdot Y_{\cdot,2} = \mathbf{w} \cdot Y_{\cdot,3} = \mathbf{w} \cdot Y_{\cdot,3n+1} = \mathbf{w} \cdot Y_{\cdot,3n+2} = \mathbf{w} \cdot Y_{\cdot,3n+3} = 1$. The reaction $R_{3n+3} = (FS_1 \longrightarrow F + S_0)$, however, returns the value $\mathbf{w} \cdot \Gamma_{\cdot,3n+3} = -1$. It follows that the left-most source complex in the $\mathbf{w}$-projected network has a reaction which directions to the left, so that (\ref{distributive}) is not endotactic by Definition \ref{endotactic} and Proposition \ref{prop:lineNetw}.
\end{proof}

Although it is known that endotacticity is a sufficient but not necessary condition for persistence of mass action systems (see the Lotka-Volterra example in Section \ref{persistencesection}), it is still somewhat surprising that the futile cycle and general $n$-site processive and distributive networks (\ref{processive}) and (\ref{distributive}), respectively, fail to be endotactic. We note therefore that intuition is often a poor guide for endotacticity, especially for higher-dimensional systems.

To recover meaningful results regarding the networks (\ref{processive}) and (\ref{distributive}), we note that such systems are commonly modeled with Michaelis-Menten kinetics through the standard quasi-steady state approximation (QSSA) \cite{M-M}. In this framework, the concentrations of the intermediate compounds (e.g. $ES_i$ and $FS_i$) are assumed to equilibriate quickly so that their derivatives may be set to zero. The conservation laws
\begin{equation}\label{eq:totals}
E_T = [E] + \sum_{i=0}^{n-1} [ES_i], \; \; \; \mbox{ and } \; \; \; F_T = [F] + \sum_{i=1}^n [FS_i]
\end{equation}
may then be used to produce algebraic expressions for $[ES_i]$ and $[FS_i]$ which can be substituted into the remaining equations (\ref{de}) to give modified formation rates for $S_i$, $i=1,\ldots, n$. The details of the analysis, including the proof of the following result, are contained in the Supplemental Material.

\begin{lemma}
\label{qssalemma}
The $n$-site processive phosphorylation / dephosphorylation network (\ref{processive}) may be modeled under the QSSA by the modified network
\begin{equation}
\label{qssaprocessive}
S_0 \mathop{\stackrel{\nu^+(S_0)}{\begin{array}{c} \vspace{-0.25cm} \longleftarrow \vspace{-0.3cm} \\ \longrightarrow \end{array}}}_{\nu^-(S_n)} S_n
\end{equation}
where
\begin{equation}
\label{kinetics1}
\nu^+(S_0) = \frac{V^+ [S_0]}{K^+ + [S_0]} \; \; \; \mbox{ and } \; \; \; \nu^-(S_n) = \frac{V^- [S_n]}{K^- + [S_n]}
\end{equation}
and $V^+, V^-, K^+, K^- > 0$ are constants which depend upon the original rate constants and total enzyme concentrations $E_T$ and $F_T$ (see \ref{eq:totals}). The $n$-site distributive phosphorylation / dephosphorylation network (\ref{distributive}) may be modeled under the QSSA by the modified network
\begin{equation}
\label{qssadistributive}
S_0 \mathop{\stackrel{\nu_1^+}{\begin{array}{c} \vspace{-0.25cm} \longleftarrow \vspace{-0.3cm} \\ \longrightarrow \end{array}}}_{\nu_1^-} S_1 \mathop{\stackrel{\nu_2^+}{\begin{array}{c} \vspace{-0.25cm} \longleftarrow \vspace{-0.3cm} \\ \longrightarrow \end{array}}}_{\nu_2^-} \cdots \mathop{\stackrel{\nu_n^+}{\begin{array}{c} \vspace{-0.25cm} \longleftarrow \vspace{-0.3cm} \\ \longrightarrow \end{array}}}_{\nu_n^-} S_n
\end{equation}
where
\begin{equation}
\label{kinetics2}
\begin{split}
\nu_i^+ & = \frac{V_i^+ [S_i]}{1 + \sum_{j=0}^{n-1} K_{ij}^+ [S_j]}, \; \; \; i=0, \ldots, n-1 \\
\nu_i^- & = \frac{V_i^- [S_i]}{1 + \sum_{j=1}^n K_{ij}^- [S_j]}, \; \; \; i=1, \ldots, n,
\end{split}
\end{equation}
where $V_i^+, K_{ij}^+ > 0$, $i, j=0, \ldots, n-1$, and $V_i^-, K_{ij}^- > 0$, $i,j=1,\ldots, n$, are constants which depend upon the original rate constants and total enzyme concentrations $E_T$ and $F_T$.
\end{lemma}


In order to apply the endotacticity results of Section \ref{knownresultssection} we require that the systems are $\kappa$-variable mass action systems. We therefore present the following result.

\begin{lemma}
\label{lemma2}
Consider the chain of reactions $\displaystyle{S_0 \stackrel{\nu_0}{\longrightarrow} S_1 \stackrel{\nu_1}{\longrightarrow} \cdots \stackrel{\nu_{n-1}}{\longrightarrow} S_n}$ where each reaction has the form
\begin{equation}
\label{generalizedmm}
\nu_i = \frac{V_i[S_i]}{1 + \sum_{j=0}^{n-1} K_{ij} [S_{i}]}
\end{equation}
where $V_i, K_{ij} > 0$, $i,j = 0, \ldots, n-1$, are constants. Suppose that each concentration $[S_i](t)$ remains bounded for all $t \geq 0$. Then there exists an $\eta > 0$ such that, for all $t \geq 0$ and all $i=0, \ldots, n-1$,
\[\eta \; [S_i](t) \leq \nu_i \big([S_0](t), \ldots, [S_{n-1}](t) \big) \leq \frac{1}{\eta} \; [S_i](t).\]
\end{lemma}

\begin{proof}
Suppose that every $[S_i](t)$, $i=0, \ldots, n-1$, is bounded from above by a value $[S_i]_M > 0$ so that $0 \leq [S_i](t) \leq [S_i]_M$ for all $i=0, \ldots, n-1$, and all $ t \geq 0$. By bounding the denominator of (\ref{generalizedmm}), it follows that
\[\frac{V_i}{1+\sum_{j=0}^{n-1} K_{ij}[S_i]_M}  [S_i](t) \leq \frac{V_i [S_i](t)}{1+\sum_{j=0}^{n-1} K_{ij} [S_i](t)} \leq V_i [S_i](t)\]
for all $t \geq 0$ and all $i=0, \ldots, n-1$. After appropriately picking $\eta > 0$ we are done.
\end{proof}

Lemma \ref{lemma2} clearly captures both the forward (phosphorylation) and backward (dephosphorylation) chains in (\ref{qssadistributive}). It also captures (\ref{qssaprocessive}) trivially by taking $n=1$. We therefore present the following result.



\begin{theorem}
The processive and distributive phosphorylation / dephosphorylation networks (\ref{qssaprocessive}) and (\ref{qssadistributive}) are permanent when modeled with kinetics (\ref{kinetics1}) and (\ref{kinetics2}), respectively.
\end{theorem}

\begin{proof}
Both (\ref{qssaprocessive}) and (\ref{qssadistributive}) consist of a single linkage class and are weakly reversible so that the networks are strongly endotactic. In order to apply Lemma \ref{lemma2}, it remains only to bound $[S_i](t)$. We notice, however, that we have the conservations $S_T = [S_0] + [S_n]$ for (\ref{qssaprocessive}) and $S_T = [S_0] + [S_1] + \cdots + [S_n]$ for (\ref{qssadistributive}). It follows Lemma \ref{lemma2} that the systems may be modeled as $\kappa$-variable mass action systems. The systems are therefore permanent by Theorem \ref{theorem2}, and we are done.
\end{proof}

\subsection{Circadian clock network}
\label{clocksection}

Circadian clocks are highly complicated biochemical oscillators which underlie an organism's ability to maintain 24-hour sleeping and feeding cycles, among other regulatory functions. These clocks are commonly entrained by such as external stimuli as light and heat but will persistent in roughly 24-hour cycles when the external stimuli are removed. The mathematical models of these clocks have been studied significantly in the past twenty years. Common features in the proposed mechanisms include posttranslational modification and negative feedback loops \cite{L-G,Pokhilko2010,Pokhilko2013}.

In this section, we focus specifically on the {\it period} (PER) / {\it timeless} (TIM) protein regulatory mechanism in {\it Drosophila} introduced by LeLoup and Goldbeter in \cite{L-G}. The network structure is given by
\[
\mbox{{\it \textbf{(Clock full)}}} \; \; \left\{ \begin{array}{l} \\ P_2 + T_2 \mathop{\stackrel{MA}{\begin{array}{c} \vspace{-0.25cm} \longleftarrow \vspace{-0.3cm} \\ \longrightarrow \end{array}}} C \mathop{\stackrel{MA}{\begin{array}{c} \vspace{-0.25cm} \longleftarrow \vspace{-0.3cm} \\ \longrightarrow \end{array}}} C_N \; \longrightarrow \; 0 \\ \\ \end{array} \hspace{-0.075in} \begin{array}{c} \\[-0.12in] \nearrow \\[0.1in] \searrow \end{array} \hspace{-0.075in} \begin{array}{l} M_P \stackrel{mix}{\longrightarrow} P_0 \mathop{\stackrel{MM}{\begin{array}{c} \vspace{-0.25cm} \longleftarrow \vspace{-0.3cm} \\ \longrightarrow \end{array}}} P_1 \mathop{\stackrel{MM}{\begin{array}{c} \vspace{-0.25cm} \longleftarrow \vspace{-0.3cm} \\ \longrightarrow \end{array}}} P_2 \\[0.15in] \\ M_T \stackrel{mix}{\longrightarrow} T_0 \mathop{\stackrel{MM}{\begin{array}{c} \vspace{-0.25cm} \longleftarrow \vspace{-0.3cm} \\ \longrightarrow \end{array}}} T_1 \mathop{\stackrel{MM}{\begin{array}{c} \vspace{-0.25cm} \longleftarrow \vspace{-0.3cm} \\ \longrightarrow \end{array}}} T_2 \end{array} \right.
\]
where we have the following constituent species:
\begin{itemize}
\item $M_P$ and $M_T$: mRNA coding for PER and TIM
\item $P_0$, $P_1$, and $P_2$: un-, mono-, and bi-phosphorylated PER
\item $T_0$, $T_1$, and $T_2$: un-, mono-, and bi-phosphorylated TIM
\item $C$ and $C_N$: PER-TIM compound (cytoplasmic and nucleic form).
\end{itemize}
The kinetic terms associated with the interactions are labeled as follows: MA = mass action, MM = Michaelis-Menten, mix = combination of mass action (target) and Michaelis-Menten (source). In addition to the interaction displayed, there is also Hill-type inhibition from $C_N$ on the input of $M_P$ and $M_T$ and linear degradation in all species (i.e. $M_i \to 0$, $C_i \to 0$, $P_i \to 0$, $T_i \to 0$). The full dynamical model is presented in the Supplemental Material.

Elements of the full clock model are also present in the papers \cite{Pokhilko2010} and \cite{Pokhilko2013} by Pokhilko {\it et al.} In particular, all three papers contain a negative feedback loop. In the full clock model, this feedback is the result of the PER-TIM compound inhibiting the transcription of the messenger RNA which codes for PER and TIM. This is shown in \cite{L-G} to lead to oscillations in the concentrations of PER and TIM where the oscillations can be made to be the order of 24 hours through parameter selection.


In order to apply the computational procedure of Section \ref{mainsection} and the theory outlined in Section \ref{endotacticnetworkssection} we require that the system can be written as a $\kappa$-variable mass action system. We therefore introduce the following reduced model:
\[
\mbox{{\it \textbf{(Clock reduced)}}} \; \; \left\{ \begin{array}{l} \\ P_2 + T_2 \mathop{\begin{array}{c} \vspace{-0.25cm} \longleftarrow \vspace{-0.3cm} \\ \longrightarrow \end{array}} C \mathop{\begin{array}{c} \vspace{-0.25cm} \longleftarrow \vspace{-0.3cm} \\ \longrightarrow \end{array}} C_N \; \longrightarrow \; 0 \\ \\ \end{array} \hspace{-0.075in} \begin{array}{c} \\[-0.18in] \nearrow \\[0.1in] \searrow \end{array} \hspace{-0.075in} \begin{array}{l} P_0 \mathop{\begin{array}{c} \vspace{-0.25cm} \longleftarrow \vspace{-0.3cm} \\ \longrightarrow \end{array}} P_1 \mathop{\begin{array}{c} \vspace{-0.25cm} \longleftarrow \vspace{-0.3cm} \\ \longrightarrow \end{array}} P_2 \\[0.2in] \\ T_0 \mathop{\begin{array}{c} \vspace{-0.25cm} \longleftarrow \vspace{-0.3cm} \\ \longrightarrow \end{array}} T_1 \mathop{\begin{array}{c} \vspace{-0.25cm} \longleftarrow \vspace{-0.3cm} \\ \longrightarrow \end{array}} T_2 \end{array} \right.
\]
\noindent We make the following notes about the reduced clock network. Firstly, all terms have $\kappa$-variable mass action form. Secondly, all species undergo linear degradation not shown in the figure for visual clarity. Thirdly, we have removed $M_P$ and $M_T$, corresponding to messenger RNAs coding for PER and TIM, from the full clock network. They have been bounded and absorbed into the rates for the influx of $P_0$ and $T_0$, respectively. Fourthly, the reduced clock network is a single linkage class network but is neither reversible nor weakly reversible. Notably, it is not weakly reversible because there is no path from any complex the right of the zero complex to any complex on the left of it.

We present the following result, the proof of which can be found in the Supplemental Material.

\begin{lemma}
\label{correspondence}
The full clock network with various kinetic rates can be represented with the reduced clock network with $\kappa$-variable mass action kinetics.
\end{lemma}


We now apply the MILP procedure outlined in Section \ref{mainsection} to prove the following result.
\begin{theorem}
\label{pertimtheorem}
The full clock network with various kinetic rates is permanent for all choices of positive parameter values.
\end{theorem}

\begin{proof}
The implementation of the strongly endotactic algorithm in {\sf CoNtRol} confirms that the reduced clock network is strongly endotactic by application of Lemma \ref{endolemma}. It follows from Theorem \ref{theorem2} (Theorem 1.1, \cite{Gopal2014}) that the network is permanent under any choice of $\kappa$-variable mass action kinetics. Since the full clock network and corresponding dynamics can be written as the reduced network with $\kappa$-variable mass action kinetics by Lemma \ref{correspondence}, we are done.
\end{proof}

Theorem \ref{pertimtheorem} suggests a general mechanism which is not weakly reversible but which is strongly endotactic. We split this mechanism into a basic and general cases, given by the following structures:
\[
\begin{array}{l}
\; \; \; \; \mbox{{\it \textbf{(Clock basic)}}} \; \; \left\{ \begin{array}{l} \\ P + T \mathop{\begin{array}{c} \vspace{-0.25cm} \longleftarrow \vspace{-0.3cm} \\ \longrightarrow \end{array}} C \; \longrightarrow \; 0 \\ \\ \end{array} \hspace{-0.075in} \begin{array}{c} \\[-0.15in] \nearrow \\[0.15in] \searrow \end{array} \hspace{-0.24in} \begin{array}{c} \\[-0.14in] \swarrow \\[0.1in] \nwarrow \end{array} \hspace{-0.075in} \begin{array}{l} P \\[0.2in] \\ T \end{array} \right. \\[0.4in]
\mbox{{\it \textbf{(Clock general)}}} \; \; \left\{ \begin{array}{l} \\ P_{n_P} + T_{n_T} \mathop{\begin{array}{c} \vspace{-0.25cm} \longleftarrow \vspace{-0.3cm} \\ \longrightarrow \end{array}} C_0 \mathop{\begin{array}{c} \vspace{-0.25cm} \longleftarrow \vspace{-0.3cm} \\ \longrightarrow \end{array}} \cdots \mathop{\begin{array}{c} \vspace{-0.25cm} \longleftarrow \vspace{-0.3cm} \\ \longrightarrow \end{array}} C_{n_C} \; \longrightarrow \; 0 \\ \\ \end{array} \hspace{-0.075in} \begin{array}{c} \\[-0.18in] \nearrow \\[0.15in] \searrow \end{array} \hspace{-0.075in} \begin{array}{l} P_0 \mathop{\begin{array}{c} \vspace{-0.25cm} \longleftarrow \vspace{-0.3cm} \\ \longrightarrow \end{array}} \cdots \mathop{\begin{array}{c} \vspace{-0.25cm} \longleftarrow \vspace{-0.3cm} \\ \longrightarrow \end{array}} P_{n_P} \\[0.2in] \\ T_0 \mathop{\begin{array}{c} \vspace{-0.25cm} \longleftarrow \vspace{-0.3cm} \\ \longrightarrow \end{array}} \cdots \mathop{\begin{array}{c} \vspace{-0.25cm} \longleftarrow \vspace{-0.3cm} \\ \longrightarrow \end{array}} T_{n_T} \end{array} \right.
\end{array}
\]
The basic clock network is presented in complete form while the general clock network contains unshown linear degradation reactions. A representation of the basic clock network in the species space of the $(P,T,C)$-plane is given in Figure \ref{generalnetwork}(A). Networks of the general clock form appear frequently in models of circadian clocks. In \cite{Pokhilko2010}, Pokhilko {\it et al.} present a model consisting of $19$ species. In \cite{Pokhilko2013}, Pokhilko {\it et al.} present a model consisting of $32$ species. Both mechanisms are shown to be strongly endotactic by the algorithm presented in Section \ref{milpsendo}.

\begin{figure}[t]
\begin{center}
\includegraphics[width=10cm]{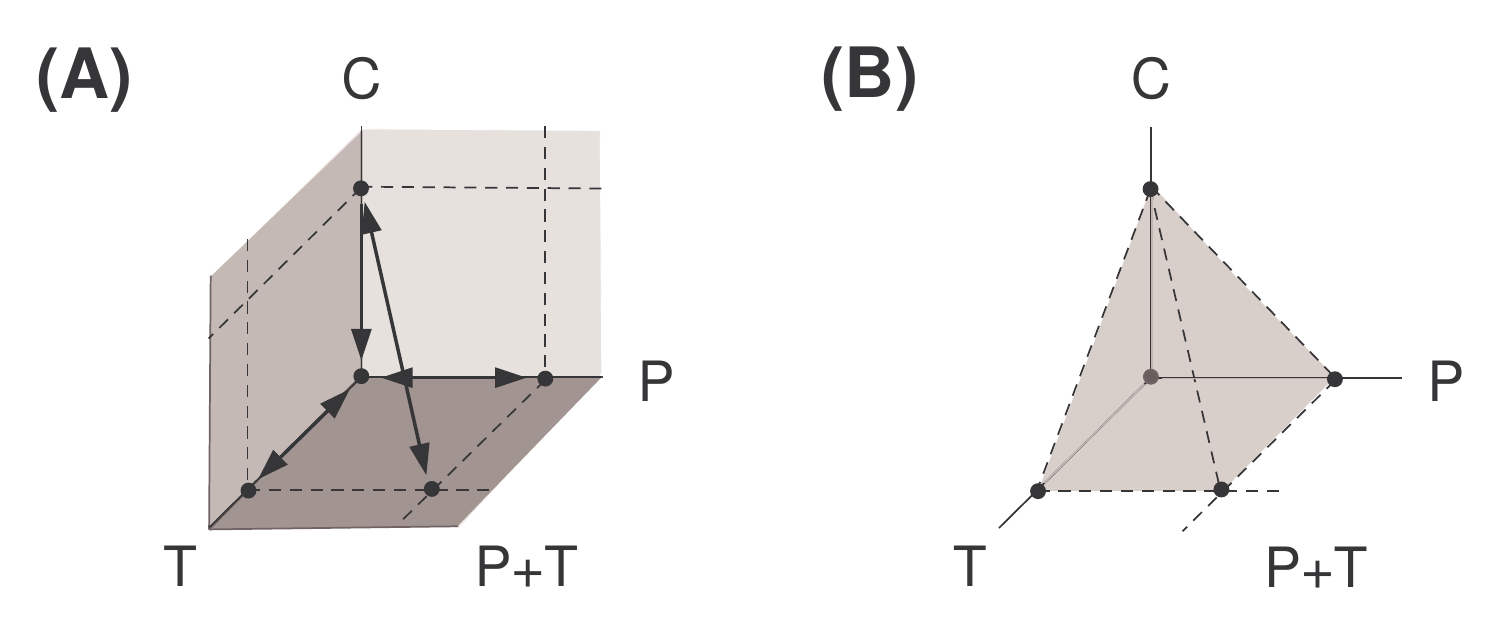}
\end{center}
\caption{\small In (A), we have a representation of the basic clock network in the species space of the $(P,T,C)$-plane. The network is not weakly reversible. In (B), we have the convex hull of the source complexes. It can be easily verified by visual inspection using the parallel sweep test that the network is strongly endotactic.}
\label{generalnetwork}
\end{figure}

We now prove that the general clock network is strongly endotactic and therefore permanent under $\kappa$-variable kinetics.

\begin{theorem}
\label{generalprop}
The general clock network is permanent when modeled with $\kappa$-variable mass action kinetics.
\end{theorem}
\begin{proof}
Let $m=(n_C+1)+(n_P+1)+(n_T+1)$ denote the number of species in the network and let $\mathbf{v} \in {\mathbb R}^m$ be picked arbitrarily. In view of Proposition \ref{prop:lineNetw}, it's enough to argue that the projection of the general clock network onto {\bf v} is endotactic.  The projected line network will be of the form ${\cal L}_1\to {\cal L}_2$, where ${\cal L}_1$ and ${\cal L}_2$ are strongly connected components and $0\in\cal L_2$. If ${\cal L}_2$ contains at least two complexes, then each complex is a source complex, and the line network is endotactic; since the network has one linkage class, it follows by \cite{Gopal2014} that in fact it is strongly endotactic. If ${\cal L}_2$ contains only one complex (necessarily the zero complex), then since both complexes $T_{n_T}$ and $P_{n_P}$ project to zero, so does the complex $P_{n_P}+T_{n_T},$ which means that $0\in{\cal L}_1.$ Then the projected line network is weakly reversible; since it has one linkage class, it is strongly endotactic.
\end{proof}

\section{Conclusion}

In this paper, we have addressed the open question of how to determine whether a given chemical reaction network is lower endotactic, endotactic, or strongly endotactic. We have presented a mixed-integer linear programming algorithm capable of establishing or disproving these properties and implemented it within the open-source online software package {\sf CoNtRol}. Notably, existing results were only able to establish these endotacticity properties for networks with at most two constituent species.

We have also applied the algorithm to classify networks with the BioModels Database. This survey has revealed a circadian clock mechanism which is strongly endotactic but not weakly reversible and therefore beyond the scope of existing theory. We have extended this mechanism and shown that this extended model is strongly endotactic. We have also considered the general $n$-site phosphorylation / dephosphorylation network in both processive and distributive forms.

This work raises several important questions for future work:
\begin{enumerate}
\item
{\it How far-reaching is the $\kappa$-variable mass action form?} We have shown that generalized Michaelis-Menten kinetics may be expressed in $\kappa$-variable kinetic form provided the substrates are bounded (Lemma \ref{lemma2}); however, it is unclear which conditions are needed for more general kinetic forms in order to make this transition. This will be the focus of future research.
\item
{\it How wide-spread is endotacticity in biochemical reaction networks?} Our search in the BioModels Database has revealed that the majority of endotactic networks are weakly reversible, and that the majority of strongly endotactic networks are weakly reversible consisting of a single linkage class. Whether this is the general trend, and whether there are more strongly endotactic non-weakly reversible biochemical mechanisms such as the ones analyzed in Section \ref{clocksection}, requires more research.
\end{enumerate}

\section*{Acknowledgments}

M.J. is supported by grant Army Research Office grant W911NF-14-1-0401. P.D. is supported by EPSRC grant EP/J00826/1. The authors thank Gheorghe Craciun for helpful discussions and suggestions, and Murad Banaji for organizing the workshop ``Combinatorial and algebraic approaches to chemical reaction networks'' at the University of Portsmouth (UK) at which this collaboration began.

\newpage


\section*{Proof of Lemma \ref{endolemma}}

In this section, we prove the equivalence of the two characterizations of endotacticity given in the main text (Definition \ref{endotactic} and Lemma \ref{endolemma}). For completeness, we restate the relevant definitions and results.

\begin{definition}[Definition \ref{endotactic}, main text]
\label{endotactic5}
A chemical reaction network is said to be:
\begin{enumerate}
\item
\textbf{endotactic} if, for every $\mathbf{w} \in \mathbb{R}^m$ and every $R_i \in \mathcal{R}$, $\mathbf{w} \cdot (y_{\rho'(i)} - y_{\rho(i)}) < 0$ implies that there is a $R_j \in \mathcal{R}$ such that $\mathbf{w} \cdot (y_{\rho(j)} - y_{\rho(i)}) < 0$ and $\mathbf{w} \cdot (y_{\rho'(j)} - y_{\rho(j)}) > 0$.
\item
\textbf{lower endotactic} if condition $1.$ holds for every $\mathbf{w} \in \mathbb{R}^m_{\geq 0}$ instead of $\mathbf{w} \in \mathbb{R}^m$.
\item
\textbf{strongly endotactic} if condition $1.$ holds and $R_j \in \mathcal{R}$ can be chosen so that $\mathbf{w} \cdot (y_{\rho(k)} - y_{\rho(j)}) \geq 0$ for all $R_k \in \mathcal{R}$.
\end{enumerate}
\end{definition}

\begin{lemma}[Lemma \ref{endolemma}, main text]
\label{endolemma5}
A chemical reaction network is not endotactic if and only if there is a non-zero vector $\mathbf{w}= (w_1,w_2,\ldots, w_m) \in \mathbb{R}^m$ and a partition of the reaction set $\mathcal{R}$ into three disjoint sets $\mathcal{R}_0$, $\mathcal{R}_-$, and $\mathcal{R}_+$ such that:
\begin{enumerate}
\item[E1.]
$\mathbf{w} \cdot (y_{\rho'(i)} - y_{\rho(i)}) = 0$ for all $R_i \in \mathcal{R}_0$;
\item[E2.]
$\mathbf{w} \cdot (y_{\rho'(i)} - y_{\rho(i)}) < 0$ for all $R_i \in \mathcal{R}_-$;
\item[E3.]
$\mathbf{w} \cdot (y_{\rho(i)} - y_{\rho(j)}) \leq 0$ for all $R_i \in \mathcal{R}_-$ and $R_j \in \mathcal{R}_+$; and
\item[E4.]
$\mathcal{R}_- \not= \O$.
\end{enumerate}
A network is not lower endotactic if and only if there is a non-zero $\mathbf{w} \in \mathbb{R}^m$ and a partition of $\mathcal{R}$ such that E1, E2, E3, and E4 are satisfied as well as:
\begin{enumerate}
\item[LE.]
$\mbox{w}_k \geq 0$ for all $k = 1, \ldots, m$.
\end{enumerate}
A network is not strongly endotactic if and only if there is a non-zero $\mathbf{w} \in \mathbb{R}^m$ and a partition of $\mathcal{R}$ such that E1, E2, and E4 are satisfied as well as:
\begin{enumerate}
\item[SE3(a).]
$\mathcal{R}_0 \not= \O$ and $\mathbf{w} \cdot (y_{\rho(i)} - y_{\rho(j)}) < 0$ for all $R_i \in \mathcal{R}_0$ and $R_j \in \mathcal{R}_- \cup \mathcal{R}_+$; or
\item[SE3(b).]
$\mathcal{R}_0 = \O$ and $\mathbf{w} \cdot (y_{\rho(i)} - y_{\rho(j)}) \leq 0$ for all $R_i \in \mathcal{R}_-$ and $R_j \in \mathcal{R}_+$; and
\end{enumerate}
\end{lemma}

\begin{proof}

We will consider the endotactic and strongly endotactic cases separately. The lower endotactic case follows with a trivial modification to the endotactic case.\\

\noindent {\it (Not endotactic $\Longrightarrow$)} Suppose that the network is not endotactic. We construct the sets $\mathcal{R}_-$, $\mathcal{R}_0$, and $\mathcal{R}_+$ in the following way: (1) $\mathcal{R}_-$ consists of those $R_i$ for which $\mathbf{w} \cdot (y_{\rho'(i)} - y_{\rho(i)}) < 0$ and $\mathbf{w} \cdot (y_{\rho(i)} - y_{\rho(j)}) \leq 0$ for all other $R_j \in \mathcal{R}$ such that $\mathbf{w} \cdot (y_{\rho'(j)} - y_{\rho(j)}) < 0$; (2) $\mathcal{R}_0$ consists of those $R_i$ for which $\mathbf{w} \cdot (y_{\rho'(i)} - y_{\rho(i)}) = 0$; and (3) $\mathcal{R}_+$ consists of everything else. We have that E1 and E2 are satisfied by construction and that E4 is satisfied by the assumption that the network is not endotactic. To show that E3 is satisfied, we suppose that there is a pair $R_i \in \mathcal{R}_-$ and $R_j \in \mathcal{R}_+$ such that $\mathbf{w} \cdot (y_{\rho(i)} - y_{\rho(j)}) > 0$. By the construction of $\mathcal{R}_-$, we have $\mathbf{w} \cdot (y_{\rho'(i)} - y_{\rho(i)}) < 0$, and by the construction of $\mathcal{R}_-$ and $\mathcal{R}_0$, it follows that $\mathbf{w} \cdot (y_{\rho'(j)} - y_{\rho(j)}) > 0$. It follows by Definition \ref{endotactic5} that the network is endotactic, which contradicts our assumption. It follows that E3 is satisfied.\\

\noindent {\it (Not endotactic $\Longleftarrow$)} Suppose there is a $\mathbf{w} \in \mathbb{R}^m$ which satisfies E1 - E4. It follows from E2 and E4 that there is a $R_i \in \mathcal{R}$ such that $\mathbf{w} \cdot (y_{\rho'(i)} - y_{\rho(i)} ) < 0$. Furthermore, from E1 and E3 we have that for every $R_j \in \mathcal{R}$, either $\mathbf{w} \cdot (y_{\rho(j)} - y_{\rho(i)}) \geq 0$ (if $R_j \in \mathcal{R}_+$) or $\mathbf{w} \cdot (y_{\rho'(j)} - y_{\rho(j)}) = 0$ (if $R_j \in \mathcal{R}_0 \cup \mathcal{R}_-$). It follows that the network is not endotactic by the contrapositive of Definition \ref{endotactic5}. The cases for lower endotactic networks follows similarly.\\

\noindent {\it (Not strongly endotactic $\Longrightarrow$)} Suppose that the network is not strongly endotactic. We construct the sets $\mathcal{R}_-$, $\mathcal{R}_0$, and $\mathcal{R}_+$ in the following way: (1) $\mathcal{R}_-$ consists of those $R_i$ for which $\mathbf{w} \cdot (y_{\rho'(i)} - y_{\rho(i)}) < 0$ and $\mathbf{w} \cdot (y_{\rho(i)} - y_{\rho(j)}) \leq 0$ for all other $R_j \in \mathcal{R}$ such that $\mathbf{w} \cdot (y_{\rho'(j)} - y_{\rho(j)}) < 0$; (2) $\mathcal{R}_0$ consists of those $R_j$ for which $\mathbf{w} \cdot (y_{\rho'(j)} - y_{\rho(j)}) = 0$ and $\mathbf{w} \cdot (y_{\rho(j)} - y_{\rho(i)}) < 0$ for all $R_i \in \mathcal{R}_-$; and (3) $\mathcal{R}_+$ consists of everything else. The condition E1, E2, and E4 are satisfied by the constructions of $\mathcal{R}_-$ and $\mathcal{R}_0$.

Now suppose that $\mathcal{R}_0 \not= \emptyset$ and there is a pair $R_i \in \mathcal{R}_0$ and $R_j \in \mathcal{R}_- \cup \mathcal{R}_+$ such that $\mathbf{w} \cdot (y_{\rho(i)} - y_{\rho(j)}) \geq 0$. $R_j \in \mathcal{R}_-$ violates the construction of $\mathcal{R}_0$ so that we must have $R_j \in \mathcal{R}_+$. By the construction of $\mathcal{R}_-$ and $\mathcal{R}_0$ we must have $\mathbf{w} \cdot (y_{\rho'(j)} - y_{\rho(j)}) > 0$ and $\mathbf{w} \cdot (y_{\rho(j)}-y_{\rho(i)}) < 0$ so that the network is endotactic by Definition \ref{endotactic5}. Now suppose that $\mathcal{R}_0 = \emptyset$ and there is a pair $R_i \in \mathcal{R}_-$ and $R_j \in \mathcal{R}_+$ such that $\mathbf{w} \cdot (y_{\rho(i)} - y_{\rho(j)}) > 0$. By the construction of $\mathcal{R}_-$ and $\mathcal{R}_0$ (which is empty), it follows that $\mathbf{w} \cdot (y_{\rho'(j)} - y_{\rho(j)}) > 0$ so that the network is endotactic by Definition \ref{endotactic5}. Since these results contradict our assumption, it follows that one of SE3(a) or SE(b) is satisfied, and we are done.\\



\noindent {\it (Not strongly endotactic $\Longleftarrow$)} Suppose there is a $\mathbf{w} \in \mathbb{R}^m$ which satisfies E1, E2, SE3(a), and E4. It follows from E2 and E4 that there is a $R_i \in \mathcal{R}$ such that $\mathbf{w} \cdot (y_{\rho'(i)} - y_{\rho(i)} ) < 0$. It follows from E1 and SE3(a) that there is an $R_i \in \mathcal{R}$ such that $\mathbf{w} \cdot (y_{\rho'(i)} - y_{\rho(i)}) = 0$ and $\mathbf{w} \cdot (y_{\rho(i)} - y_{\rho(j)}) < 0$ for every $R_j \in \mathcal{R}$ such that $\mathbf{w} \cdot (y_{\rho'(i)} - y_{\rho(i)}) < 0$. It follows by Definition \ref{endotactic5} that the network is not strongly endotactic.

Now suppose there is a $\mathbf{w} \in \mathbb{R}^m$ which satisfies E1, E2, SE3(b), and E4. It follows from E2 and E4 that there is a $R_i \in \mathcal{R}$ such that $\mathbf{w} \cdot (y_{\rho'(i)} - y_{\rho(i)} ) < 0$. It follows from E1 and SE3(b) that $\mathbf{w} \cdot (y_{\rho(i)} - y_{\rho(j)} \geq 0$ for every $R_j \in \mathcal{R}$ such that $\mathbf{w} \cdot (y_{\rho'(j)} - y_{\rho(j)}) > 0$ so that the network is not endotactic by Definition \ref{endotactic5} and therefore not strongly endotactic. This completes the cases, so we are done.
\end{proof}

\section*{Proof of Lemma \ref{qssalemma}}

In the main text, we introduced two general phosphorylation / dephosphorylation mechanisms which appear frequently in models of biochemical reaction systems. The general $n$-site distributive phosphorylation network is given by
\begin{equation}
\label{processive1}
\left\{
\begin{array}{c}
\displaystyle{S_0 + E \mathop{\stackrel{k_1}{\begin{array}{c} \vspace{-0.25cm} \longleftarrow \vspace{-0.3cm} \\ \longrightarrow \end{array}}}_{k_2} ES_0 \mathop{\stackrel{k_3}{\begin{array}{c} \vspace{-0.25cm} \longleftarrow \vspace{-0.3cm} \\ \longrightarrow \end{array}}}_{k_4} ES_1 \mathop{\stackrel{k_5}{\begin{array}{c} \vspace{-0.25cm} \longleftarrow \vspace{-0.3cm} \\ \longrightarrow \end{array}}}_{k_6} \cdots \mathop{\stackrel{k_{2n-1}}{\begin{array}{c} \vspace{-0.25cm} \longleftarrow \vspace{-0.3cm} \\ \longrightarrow \end{array}}}_{k_{2n}} ES_{n-1} \; \stackrel{k_{2n+1}}{\longrightarrow} \; S_n + E} \vspace{0.1in} \\
\displaystyle{S_n + F \mathop{\stackrel{\ell_{2n+1}}{\begin{array}{c} \vspace{-0.25cm} \longleftarrow \vspace{-0.3cm} \\ \longrightarrow \end{array}}}_{\ell_{2n}} FS_n \mathop{\stackrel{\ell_{2n-1}}{\begin{array}{c} \vspace{-0.25cm} \longleftarrow \vspace{-0.3cm} \\ \longrightarrow \end{array}}}_{\ell_{2n-2}} \cdots \mathop{\stackrel{\ell_5}{\begin{array}{c} \vspace{-0.25cm} \longleftarrow \vspace{-0.3cm} \\ \longrightarrow \end{array}}}_{\ell_4} FS_2 \mathop{\stackrel{\ell_3}{\begin{array}{c} \vspace{-0.25cm} \longleftarrow \vspace{-0.3cm} \\ \longrightarrow \end{array}}}_{\ell_2} FS_1 \; \stackrel{\ell_1}{\longrightarrow} \; S_0 + F}
\end{array}
\right.
\end{equation}
while the $n$-site distributive phosphorylation network is given by
\begin{equation}
\label{distributive1}
\left\{
\begin{split}
& \displaystyle{S_0 + E \mathop{\stackrel{k_1}{\begin{array}{c} \vspace{-0.25cm} \longleftarrow \vspace{-0.3cm} \\ \longrightarrow \end{array}}}_{k_2} S_0E \; \stackrel{k_3}{\longrightarrow} S_1 + E \mathop{\stackrel{k_4}{\begin{array}{c} \vspace{-0.25cm} \longleftarrow \vspace{-0.3cm} \\ \longrightarrow \end{array}}}_{k_5} S_1E \; \stackrel{k_6}{\longrightarrow} \; \; \; \cdots \; \; \; \stackrel{k_{3n-3}}{\longrightarrow} S_{n-1} + E \mathop{\stackrel{k_{3n-2}}{\begin{array}{c} \vspace{-0.25cm} \longleftarrow \vspace{-0.3cm} \\ \longrightarrow \end{array}}}_{k_{3n-1}} S_{n-1}E\stackrel{k_{3n}}{\longrightarrow} S_n + E} \\
&\displaystyle{S_n + F \mathop{\stackrel{\ell_{3n-2}}{\begin{array}{c} \vspace{-0.25cm} \longleftarrow \vspace{-0.3cm} \\ \longrightarrow \end{array}}}_{\ell_{3n-1}} S_nF \; \stackrel{\ell_{3n}}{\longrightarrow} S_{n-1} + F \mathop{\stackrel{\ell_{3n-5}}{\begin{array}{c} \vspace{-0.25cm} \longleftarrow \vspace{-0.3cm} \\ \longrightarrow \end{array}}}_{\ell_{3n-4}} S_{n-1}F \; \stackrel{\ell_{3n-3}}{\longrightarrow}\; \; \; \cdots \; \; \; \stackrel{\ell_6}{\longrightarrow} S_{1} + F \mathop{\stackrel{\ell_1}{\begin{array}{c} \vspace{-0.25cm} \longleftarrow \vspace{-0.3cm} \\ \longrightarrow \end{array}}}_{\ell_2} S_1F \; \stackrel{\ell_3}{\longrightarrow} S_0 + F}
\end{split}
\right.
\end{equation}

In this section, we derive the general form of the standard quasi-steady state approximation (QSSA) of (\ref{processive1}) and (\ref{distributive1}). Under the QSSA, we assume that the intermediate compounds (i.e. the species $S_iE$ and $S_iF$) equilibriate on a sufficiently short time-scale that their concentrations may be assumed to be constant. To accommodation this assumption, we set their derivatives equal to zero, so that we have
\begin{equation}
\label{qssa}
\begin{split}
\dot{[S_iE]} & = 0, \; \; \; i=0, \ldots, n-1 \\
\dot{[S_iF]} & = 0, \; \; \; i=1, \ldots, n.
\end{split}
\end{equation}
We may use (\ref{qssa}) in conjunction with the conservation laws
\begin{equation}
\label{conservation}
E_T = [E] + \sum_{i=0}^{n-1} [ES_i], \; \; \; \mbox{ and } \; \; \; F_T = [F] + \sum_{i=1}^n [FS_i]
\end{equation}
to determine algebraic expressions for $[S_iE]$ and $[S_iF]$. The rate of each phosphorylation or dephosphorylation step in (\ref{processive1}) or (\ref{distributive1}) may then be simplified to a single step which depends upon the equilibriated value of the corresponding step's intermediate compound.

The form of the QSSA rates for (\ref{processive1}) and (\ref{distributive1}) are well-known for the cases $n=0$ (one-step chain) and $n=1$ (two-step chain). To the best of the authors' knowledge, the derivation contained here is the first to determine the QSSA rate form for the general $n$-site chains (\ref{processive1}) and (\ref{distributive1}). We have the following results which we prove by direct substitution. The forms may be derived directly by induction.

\begin{lemma}
\label{lemma5}
The processive phosphorylation chain
\begin{equation}
\label{processive5}
\displaystyle{S_0 + E \mathop{\stackrel{k_1}{\begin{array}{c} \vspace{-0.25cm} \longleftarrow \vspace{-0.3cm} \\ \longrightarrow \end{array}}}_{k_2} ES_0 \mathop{\stackrel{k_3}{\begin{array}{c} \vspace{-0.25cm} \longleftarrow \vspace{-0.3cm} \\ \longrightarrow \end{array}}}_{k_4} ES_1 \mathop{\stackrel{k_5}{\begin{array}{c} \vspace{-0.25cm} \longleftarrow \vspace{-0.3cm} \\ \longrightarrow \end{array}}}_{k_6} \cdots \mathop{\stackrel{k_{2n-1}}{\begin{array}{c} \vspace{-0.25cm} \longleftarrow \vspace{-0.3cm} \\ \longrightarrow \end{array}}}_{k_{2n}} ES_{n-1} \; \stackrel{k_{2n+1}}{\longrightarrow} \; S_n + E}
\end{equation}
has the QSSA solution
\begin{equation}
\label{processivesolution}
[ES_{i}] = \displaystyle{\frac{\left( \prod_{j=0}^iK_j \right) E_T [S_0]}{1 + \left( \sum_{k=0}^{n-1} \prod_{j=0}^{k-1} K_j\right) [S_0]}}, \; \; \; i=0, \ldots, n-1
\end{equation}
where $K_i = k_{2i+1}/(k_{2i+2}+k_{2i+3})$, $i=0, \ldots, n-1$.
\end{lemma}

\begin{proof}
We need to solve the system of equations
\begin{eqnarray}
\label{11}
\dot{[S_0E]} & = & k_1[S_0][E] - (k_2+k_3)[ES_0] = 0\\
\label{12}
\dot{[S_iE]} & = & k_{2i+1}[ES_{i-1}] - (k_{2i+2}+k_{2i+3})[ES_i] = 0, \; \; \; i=1, \ldots, n-1\\
\label{13}
E_T & = & [E] + \sum_{i=0}^{n-1} [ES_i].
\end{eqnarray}
We substitute (\ref{13}) into (\ref{11}) and use the substitution $K_i = k_{2i+1}/(k_{2i+2}+k_{2i+3})$ for $i=0, \ldots, n-1$, to get the new system of equations
\begin{eqnarray}
\label{14}
& & K_0[S_0]\left(E_T - \sum_{i=0}^{n-1} [ES_i] \right) - [ES_0] = 0\\
\label{15}
& & K_{i}[ES_{i-1}] - [ES_i] = 0, \; \; \; i=1, \ldots, n-1
\end{eqnarray}
Substituting (\ref{processivesolution}) into the left-hand size of (\ref{14}) gives
\footnotesize
\[\begin{split}
& K_0[S_0] \left( E_T - \sum_{i=0}^{n-1} \left( \frac{\left(\prod_{j=0}^{i} K_{j} \right) E_T[S_0]}{1 + \left( \sum_{k=0}^{n-1} \prod_{j=0}^{k} K_j \right) [S_0]}\right) \right)- \frac{K_0 E_T [S_0]}{1 + \left( \sum_{k=0}^{n-1} \prod_{j=0}^{k} K_j \right) [S_0]}\\
& \; \; \; = K_0 E_T [S_0] \left( \frac{1 + \left( \sum_{k=0}^{n-1} \prod_{j=0}^{k} K_j \right) [S_0] -  \left(\sum_{i=0}^{n-1} \prod_{j=0}^{i} K_{j+1} \right) [S_0]}{1 + \left( \sum_{k=0}^{n-1} \prod_{j=0}^{k} K_j \right) [S_0]}\right)- \frac{K_0 E_T [S_0]}{1 + \left( \sum_{k=0}^{n-1} \prod_{j=0}^{k} K_j \right) [S_0]}\\
& \; \; \; = \frac{K_0 E_T [S_0]}{1 + \left( \sum_{k=0}^{n-1} \prod_{j=0}^{k} K_j \right) [S_0]} - \frac{K_0 E_T [S_0]}{1 + \left( \sum_{k=0}^{n-1} \prod_{j=0}^k K_j \right) [S_0]} = 0.
\end{split}\]
\small
Substituting (\ref{processivesolution}) into the left-hand side of (\ref{15}) gives
\footnotesize
\[K_{i} \displaystyle{\frac{\left( \prod_{j=0}^{i-1}K_j \right) E_T [S_0]}{1 + \left( \sum_{k=0}^{n-1} \prod_{j=0}^k K_j\right) [S_0]}} - \displaystyle{\frac{\left( \prod_{j=1}^{i}K_j \right) E_T [S_0]}{1 + \left( \sum_{k=0}^{n-1} \prod_{j=0}^k K_{j}\right) [S_0]}} = 0\]
\small
and we are done.

\end{proof}

\begin{lemma}
\label{lemma6}
The distributive phosphorylation chain
\begin{equation}
\label{distributive5}
\displaystyle{S_0 + E \mathop{\stackrel{k_1}{\begin{array}{c} \vspace{-0.25cm} \longleftarrow \vspace{-0.3cm} \\ \longrightarrow \end{array}}}_{k_2} S_0E \; \stackrel{k_3}{\longrightarrow} S_1 + E \mathop{\stackrel{k_4}{\begin{array}{c} \vspace{-0.25cm} \longleftarrow \vspace{-0.3cm} \\ \longrightarrow \end{array}}}_{k_5} S_1E \; \stackrel{k_6}{\longrightarrow} \; \; \; \cdots \; \; \; \stackrel{k_{3n-3}}{\longrightarrow} S_{n-1} + E \mathop{\stackrel{k_{3n-2}}{\begin{array}{c} \vspace{-0.25cm} \longleftarrow \vspace{-0.3cm} \\ \longrightarrow \end{array}}}_{k{3n-1}} S_{n-1}E\stackrel{k_{3n}}{\longrightarrow} S_n + E}
\end{equation}
has the QSSA solution
\begin{equation}
\label{distributivesolution}
[ES_{i}] = \displaystyle{\frac{K_{i}E_T[S_{i}]}{1 + \sum_{j=0}^{n-1} K_j[S_{j}]}}, \; \; \; i=0,\ldots, n-1
\end{equation}
where $K_i = k_{3i+1}/(k_{3i+2}+k_{3i+3})$ for $i=0, \ldots, n-1$.
\end{lemma}

\begin{proof}
We need to solve the system of equations
\begin{eqnarray}
\label{19}
\dot{[S_iE]}& = & k_{3i+1} [S_i][E] - (k_{3i+2} + k_{3i+3})[ES_i] = 0, \; \; \; i=0, \ldots, n-1 \\
\label{16}
E_T & = & [E] + \sum_{i=0}^{n-1} [ES_i].
\end{eqnarray}
We substitute (\ref{16}) into (\ref{19}) and use the substitution $K_i = k_{3i+1}/(k_{3i+2}+k_{3i+3})$ to obtain the new system of equations
\begin{equation}
\label{10}
K_{i} [S_i] \left( E_T - \sum_{i=0}^{n-1} [ES_i]\right) - [ES_i] = 0, \; \; \; i=0, \ldots, n-1.
\end{equation}
Substituting (\ref{distributivesolution}) into the left-hand side of (\ref{10}) gives
\footnotesize
\[\begin{split}
& K_{i} [S_i] \left( E_T - \sum_{i=0}^{n-1} \frac{K_{i}E_T[S_{i}]}{1 + \sum_{j=0}^{n-1} K_j[S_{j}]}\right) - \frac{K_{i}E_T[S_{i}]}{1 + \sum_{j=0}^{n-1} K_j[S_{j}]}\\
& \; \; \; = K_i E_T [S_i] \left( \frac{1 - \sum_{j=0}^{n-1} K_j[S_{j}] + \sum_{i=0}^{n-1} K_i[S_{i}]}{1 + \sum_{j=0}^{n-1} K_j[S_{j}]} \right) - \frac{K_{i}E_T[S_{i}]}{1 + \sum_{j=0}^{n-1} K_j[S_{j}]}\\
& \; \; \; = \frac{K_{i}E_T[S_{i}]}{1 + \sum_{j=0}^{n-1} K_j[S_{j}]} - \frac{K_{i}E_T[S_{i}]}{1 + \sum_{j=0}^{n-1} K_j[S_{j}]} = 0
\end{split}\]
\small
and we are done.
\end{proof}

We now prove the following result, which was Lemma \ref{qssalemma} in the main text. We note that the analysis contained in the main text does not depend on the exact form of the constants but that they may be determined directly from Lemma \ref{lemma5} and Lemma \ref{lemma6}.

\begin{lemma}[Lemma \ref{qssalemma}, main text]
\label{lemma}
The $n$-site processive phosphorylation / dephosphorylation network (\ref{processive5}) may be modeled under the QSSA by the modified network
\begin{equation}
\label{qssaprocessive}
S_0 \mathop{\stackrel{\nu^+(S_0)}{\begin{array}{c} \vspace{-0.25cm} \longleftarrow \vspace{-0.3cm} \\ \longrightarrow \end{array}}}_{\nu^-(S_n)} S_n
\end{equation}
where
\begin{equation}
\label{kinetics1}
\nu^+(S_0) = \frac{V^+ [S_0]}{K^+ + [S_0]} \; \; \; \mbox{ and } \; \; \; \nu^-(S_n) = \frac{V^- [S_n]}{K^- + [S_n]}
\end{equation}
where $V^+, V^-, K^+, K^- > 0$ are constants which depend upon the original rate constants. The $n$-site distributive phosphorylation / dephosphorylation network (\ref{distributive5}) may be modeled under the QSSA by the modified network
\begin{equation}
\label{qssadistributive}
S_0 \mathop{\stackrel{\nu_1^+}{\begin{array}{c} \vspace{-0.25cm} \longleftarrow \vspace{-0.3cm} \\ \longrightarrow \end{array}}}_{\nu_1^-} S_1 \mathop{\stackrel{\nu_2^+}{\begin{array}{c} \vspace{-0.25cm} \longleftarrow \vspace{-0.3cm} \\ \longrightarrow \end{array}}}_{\nu_2^-} \cdots \mathop{\stackrel{\nu_n^+}{\begin{array}{c} \vspace{-0.25cm} \longleftarrow \vspace{-0.3cm} \\ \longrightarrow \end{array}}}_{\nu_n^-} S_n
\end{equation}
where
\begin{equation}
\label{kinetics2}
\begin{split}
\nu_i^+ & = \frac{V_i^+ [S_i]}{1 + \sum_{j=0}^{n-1} K_{ij}^+ [S_j]}, \; \; \; i=0, \ldots, n-1 \\
\nu_i^- & = \frac{V_i^- [S_i]}{1 + \sum_{j=1}^n K_{ij}^- [S_j]}, \; \; \; i=1, \ldots, n,
\end{split}
\end{equation}
where $V_i^+, K_{ij}^+ > 0$, $i, j=0, \ldots, n-1$, and $V_i^-, K_{ij}^- > 0$, $i,j=1,\ldots, n$, are constants which depend upon the original rate constants.
\end{lemma}

\begin{proof}
The form (\ref{qssaprocessive}) follows by applying Lemma \ref{lemma5} independently to the phosphorylation and dephosphorylation chains in (\ref{processive5}), then rearranging the constants in the denominator. The form (\ref{qssadistributive}) follows similarly by applying Lemma \ref{lemma6} to the chains in (\ref{distributive5}), and we are done.
\end{proof}

\section*{Proof of Lemma \ref{correspondence}}

The following is the full model for the PER {\it (period)} and TIM {\it (timeless)} protein system in {\it Drosophila} presented by Leloup and Goldbeter in \cite{L-G}:

\footnotesize
\begin{eqnarray}
\label{pertim}
\frac{dM_P}{dt} & = & \nu_{sP} \frac{K_{IP}^n}{K_{IP}^n + C_N^n} - \nu_{mP} \frac{M_P}{K_{mP}+M_P} - k_d M_P \\
\label{p0}
\frac{dP_0}{dt} & = & k_{sP} M_P - V_{1P} \frac{P_0}{K_{1P}+P_0} + V_{2P} \frac{P_1}{K_{2P}+P_1}-k_d P_0 \\
\frac{dP_1}{dt} & = & V_{1P} \frac{P_0}{K_{1P}+P_0} - V_{2P} \frac{P_1}{K_{2P}+P_1} - V_{3P} \frac{P_1}{K_{3P}+P_1} + V_{4P} \frac{P_2}{K_{4P}+P_2} - k_d P_1\\
\label{p2}
\frac{dP_2}{dt} & = & V_{3P} \frac{P_1}{K_{3P}+P_1} - V_{4P} \frac{P_2}{K_{4P}+P_2} -k_3P_2T_2 + k_4C - \nu_{dP} \frac{P_2}{K_{dP}+P_2}-k_d P_2\\
\frac{dM_T}{dt} & = & \nu_{sT} \frac{K_{IT}^n}{K_{IT}^n + C_N^n} - \nu_{mT} \frac{M_T}{K_{mT}+M_T} - k_d M_T\\
\label{t0}
\frac{dT_0}{dt} & = & k_{sT}M_T - V_{1T} \frac{T_0}{K_{1T}+T_0} + V_{2T} \frac{T_1}{K_{2T} + T_1}-k_dT_0 \\
\frac{dT_1}{dt} & = & V_{1T} \frac{T_0}{K_{1T}+T_0} - V_{2T} \frac{T_1}{K_{2T} + T_1} - V_{3T} \frac{T_1}{K_{3T} + T_1} + V_{4T} \frac{T_2}{K_{4T} + T_2} - k_dT_1 \\
\label{t2}
\frac{dT_2}{dt} & = & V_{3T} \frac{T_1}{K_{3T} + T_1} - V_{4T} \frac{T_2}{K_{4T}+T_2} - k_3P_2T_2 + k_4 C - \nu_{dT} \frac{T_2}{k_{dT}+T_2} - k_d T_2 \\
\frac{dC}{dt} & = & k_3 P_2 T_2 -k_4C - k_1C + k_2C_N -k_{dC}C \\
\label{pertim-end}
\frac{dC_n}{dt} & = & k_1C -k_2 C_N -k_{dN}C_N.
\end{eqnarray}
\small

We now prove the following result, which is equivalent to Lemma 5.4 of the main text.

\begin{lemma}[Lemma \ref{correspondence}, main text]
\label{correspondence5}
The system of equations (\ref{pertim}) - (\ref{pertim-end}) can be represented by the following network with $\kappa$-variable mass action kinetics:
\[
\mbox{{\it \textbf{(Clock reduced)}}} \; \; \left\{ \begin{array}{l} \\ P_2 + T_2 \mathop{\begin{array}{c} \vspace{-0.25cm} \longleftarrow \vspace{-0.3cm} \\ \longrightarrow \end{array}} C \mathop{\begin{array}{c} \vspace{-0.25cm} \longleftarrow \vspace{-0.3cm} \\ \longrightarrow \end{array}} C_N \; \longrightarrow \; 0 \\ \\ \end{array} \hspace{-0.075in} \begin{array}{c} \\[-0.18in] \nearrow \\[0.1in] \searrow \end{array} \hspace{-0.075in} \begin{array}{l} P_0 \mathop{\begin{array}{c} \vspace{-0.25cm} \longleftarrow \vspace{-0.3cm} \\ \longrightarrow \end{array}} P_1 \mathop{\begin{array}{c} \vspace{-0.25cm} \longleftarrow \vspace{-0.3cm} \\ \longrightarrow \end{array}} P_2 \\[0.2in] \\ T_0 \mathop{\begin{array}{c} \vspace{-0.25cm} \longleftarrow \vspace{-0.3cm} \\ \longrightarrow \end{array}} T_1 \mathop{\begin{array}{c} \vspace{-0.25cm} \longleftarrow \vspace{-0.3cm} \\ \longrightarrow \end{array}} T_2 \end{array} \right.
\]
where each species also undergoes linear degradation (i.e. $P_i \to 0$, $T_i \to 0$, $C_i \to 0$).
\end{lemma}

\begin{proof}
We split the proof into the following steps:
\begin{enumerate}
\item
We prove that all trajectories are eventually bounded by a common bound. Our method is to construct a linear functional in all of the species which decreases along trajectories. This allows us to bound the Michaelis-Menten terms in (\ref{p0}) - (\ref{p2}) and (\ref{t0}) - (\ref{t2}).
\item
We bound concentrations of $M_P$ and $M_T$ from below in the limit as $t \to \infty$ so that, after a finite time, we may incorporate the concentrations of $M_P$ and $M_T$ into the rate constants for the creation of $P_0$ and $T_0$. The bounds guarantee that this can be done with a $\kappa$-variable reaction rate.
\item
We show that this resulting system corresponds to a $\kappa$-variable mass-action system with underlying network structure given by the reduced clock network.
\end{enumerate}

\noindent {\it 1. Trajectories are bounded:} We define the linear functional
\begin{equation}
\label{linear}
\begin{split}
& H(M_P,P_0,P_1,P_2,M_T,T_0,T_1,T_2,C,C_N) \\ & \; \; \; = (k_{sP}+1) M_p + (k_{sT}+1) M_T + k_d (P_0 + P_1 + P_2 + T_0 + T_1 + T_2) + 2k_d (C + C_N).
\end{split}
\end{equation}
and introduce the bound values $M_P^* = M_T^* = k^*/k_d$, $P_0^* = P_1^* = P_2^* = T_0^* = T_1^* = T_2^* = k^* / k_d^2$, $C^* = k^* / (2k_d k_{dC}),$ and $C_N^* = k^* / (2k_d k_{dN})$ where $k^* = (k_{sP} + 1) \nu_{sP} + (k_{sT} + 1) \nu_{sT}$. We now show that, along trajectories of (\ref{pertim}) - (\ref{pertim-end}), there is a finite time $\tau > 0$ such that $H(t) \leq H^*$ for $t \geq \tau$ where
\begin{equation}
\label{hstar}
H^* = (k_{sP}+1) M_p^* + (k_{sT}+1) M_T^* + k_d (P_0^* + P_1^* + P_2^* + T_0^* + T_1^* + T_2^*) + 2k_d (C^* + C_N^*).
\end{equation}

Consider the time derivative of (\ref{linear}) along trajectories of (\ref{pertim}) - (\ref{pertim-end}). After simplification, we have
\begin{equation}
\label{derivative}
\begin{split}
\frac{dH(t)}{dt} & = (k_{sP} + 1) \nu_{sP} \frac{K_{IP}^n}{K_{IP}^n + C_N^n} + (k_{sT}+1) \nu_{sT} \frac{K_{IT}^n}{K_{IT}^n + C_N^n}\\ & \; \; \; - (k_{sP}+1)\nu_{mP} \frac{M_P}{K_{mP}+M_P} -(k_{sT}+1) \nu_{mT} \frac{M_T}{K_{mT} + M_T} - k_d \nu_{dP} \frac{P_2}{K_{dP}+P_2} - k_d \nu_{dT} \frac{T_2}{K_{dT}+T_2} \\ & \; \; \; -k_d (M_P + M_T) - k_d^2 (P_0 + P_1 + P_2 + T_0 + T_1 + T_2) - 2k_d ( k_{dC}C + k_{dN}C_N).
\end{split}
\end{equation}
Now suppose that $H(t) \geq H^*$. It follows by construction that at least one of the reactants obtains or exceeds its corresponding bound value so that the negative values in (\ref{derivative}) exceed $k^*$ in magnitude. Noting that
\[(k_{sP} + 1) \nu_{sP} \frac{K_{IP}^n}{K_{IP}^n + C_N^n} + (k_{sT}+1) \nu_{sT} \frac{K_{IT}^n}{K_{IT}^n + C_N^n} \leq (k_{sP} + 1) \nu_{sP} +(k_{sT} + 1) \nu_{sT} = k^*\]
we have that (\ref{derivative}) is nonpositive whenever $H(t) \geq H^*$. The strict negativity of the remaining negative terms (\ref{derivative}) ensures that (\ref{derivative}) is bound strictly, so that
\begin{equation}
\label{boundstuff}
\frac{dH}{dt} < -\epsilon
\end{equation}
for some $0 < \epsilon$. Integrating (\ref{boundstuff}) shows that there is a $\tau > 0$ such that $H(t) \leq H^*$ for all $t \geq \tau$. Since all of the coefficients of the reactants in $H(t)$ are positive, it follows that every reactant is eventually bounded by the common bound $H^*$.\\

\noindent {\it 2. $M_P$ and $M_T$ are bounded away from zero:} We first define
\begin{equation}
\begin{split}
M_P^* & = \frac{\nu_{sP} K_{IP}^n}{K_{IP}^n + (C_N^*)^n} \frac{k_{mP}}{\nu_{mP} + k_d k_{mP}} \\
M_T^* & = \frac{\nu_{sT} K_{IT}^n}{K_{IT}^n + (C_N^*)^n} \frac{k_{mT}}{\nu_{mT} + k_d k_{mT}}.
\end{split}
\end{equation}


\noindent Now consider the equation (\ref{pertim}) for the concentration of $M_P(t)$. Since all trajectories of the system (\ref{pertim}) - (\ref{pertim-end}) converge to a bounded set, we have that there is a $C_N^* > 0$ and a $\tau \geq 0$ such that $C_N(t) \leq C_N^*$ for all $t \geq \tau$. It follows from this and that fact that $M_P(t) \geq 0$ that, for $t \geq \tau$, we have
\begin{equation}
\label{9292}
\begin{split}
\frac{dM_P(t)}{dt} & = \nu_{sP} \frac{K_{IP}^n}{K_{IP}^n + C_N^n} - \nu_{mP} \frac{M_P}{K_{mP}+M_P} - k_d M_P \\
& \geq \nu_{sP} \frac{K_{IP}^n}{K_{IP}^n + (C_N^*)^n} - \left( \frac{\nu_{mP}}{K_{mP}} + k_d \right) M_P.
\end{split}
\end{equation}
It follows from (\ref{9292}) that
\[M_P(t) \geq M_P^* + ( M_P(0) - M_P^*) e^{-  \frac{\nu_{mP} + k_d K_{mP}}{k_{mP}} t}.\]
It follows that $\displaystyle{\liminf_{t \to \infty} MP(t) \geq M_P^*}$. A similar argument shows that $\displaystyle{\liminf_{t \to \infty} MT(t) \geq M_T^*}$, and we are done.\\

\noindent {\it 3. Construction of network:} We consider equations (\ref{p0}) - (\ref{p2}) and (\ref{t0}) - (\ref{pertim-end}). We have shown that all variables are bounded, so that Lemma \ref{lemma6} implies any Michaelis-Menten form in these variables may be written in $\kappa$-variable form. We also have shown that there is are $M_P^{*} \geq 0, M_P^{**} \geq 0, M_T^{*} \geq 0,$ $M_T^{**} \geq 0,$ and $\tau \geq 0$ such that
\[0 < M_P^{*} \leq M_P(t) \leq M_P^{**} \; \; \; \mbox{ and } \; \; \; 0 < M_T^{*} \leq M_T(t) \leq M_T^{**}\]
for $t \geq \tau$. We may therefore incorporate these variables into a $\kappa$-variable influx rate for $P_0$ and $T_0$ respectively. The remaining terms correspond to linear degradation terms which are already in $k$-variable form. 

This yields a network which is $\kappa$-variable form. For example, we may write (\ref{p0}) as
\[\frac{dP_0}{dt} = \kappa_{sP}(t) - (\kappa_{1P}(t) + \kappa_d(t)) P_0 + \kappa_{2P}(t) P_1\]
where
\[\kappa_{sP}(t) \in [ M_P^*, M_P^{**} ], \kappa_{1P}(t) \in \left[ \frac{V_{1P}}{K_{1P}+P_0^*}, \frac{V_{1P}}{K_{1P}} \right], \kappa_{2P} \in \left[ \frac{V_{2P}}{K_{2P}+P_1^*}, \frac{V_{2P}}{K_{2P}} \right], \mbox{ and } \kappa_d(t) = k_d.\]
Since the resulting network has the form of the reduced clock network, we are done.

\end{proof}

\end{document}